\documentclass[ssy,authoryear,preprint,12pt]{imsart}
\usepackage[a4paper,bindingoffset=0.5cm,left=2cm,right=2cm,top=2cm,bottom=2cm,footskip=.8cm]{geometry}
\usepackage{amssymb,amsmath,amsthm,color,graphicx,enumerate,hyperref}
\usepackage{cancel}

\renewcommand{\bar}{\overline}
\renewcommand{\hat}{\widehat}
\newcommand{\e}{{\rm e}}
\renewcommand{\d}{{\rm d}}

\renewcommand{\P}{\mathbb{P}}

\definecolor{airforceblue}{rgb}{0.36, 0.54, 0.66}

\newcommand{\BP}[1]{\textcolor{airforceblue}{#1}}

\newtheorem{theorem}{Theorem}[section]
\newtheorem{proposition}[theorem]{Proposition}
\newtheorem{lemma}[theorem]{Lemma}

\newtheorem{definition}[theorem]{Definition}

\newtheorem{remark}[theorem]{Remark}

\newtheorem{algorithm}[theorem]{Algorithm}
\newtheorem{assumption}[theorem]{Assumption}

\startlocaldefs
\usepackage{tikz}
\usetikzlibrary{arrows, automata,positioning,calc,shapes,decorations.pathreplacing,decorations.markings,shapes.misc,petri,topaths}
\usepackage{tkz-berge}
  \usepackage{pgfplots}
  \pgfplotsset{compat=newest}
  \usetikzlibrary{plotmarks}
  \usepackage{grffile}
\newlength\figureheight
  \newlength\figurewidth
  \newlength\figH
  \newlength\figW
\pgfplotsset{%
    tick label style={font=\scriptsize},
    label style={font=\footnotesize},
    legend style={font=\footnotesize},
   	every axis plot/.append style={very thick}
}

\endlocaldefs

\begin{document}

\begin{frontmatter}

\title{Detecting Markov Chain Instability:\\A Monte Carlo Approach}
\runtitle{Detecting Markov Chain Instability}


\author{\fnms{M. Mandjes$^{*}$, }\ead[label=e1]{m.r.h.mandjes@uva.nl}\snm{B. Patch$^{\dag*}$ \&}\ead[label=e2]{b.patch@uq.edu.au} \snm {N.S. Walton$^{*\circ}$}\ead[label=e3]{neil.walton@manchester.ac.uk}}
\address{\printead{e1}}
\address{ \printead{e2}}
\address{\printead{e3}}
\affiliation{$^*$University of Amsterdam}
\affiliation{$^\dag$The University of Queensland}
\affiliation{$^\circ$The University of Manchester}

\runauthor{M. Mandjes, B. Patch and N. S. Walton}

\begin{abstract}

We devise a Monte Carlo based method for detecting whether a non-negative Markov chain is stable for a given set of parameter values. More precisely, for a given subset of the parameter space, we develop an algorithm that is capable of deciding whether the set has a subset of positive Lebesgue measure for which the Markov chain is unstable. The approach is based on a variant of simulated annealing, and consequently only mild assumptions are needed to obtain performance guarantees.

The theoretical underpinnings of our algorithm are based on a result stating that the stability of a set of parameters can be phrased in terms of the stability of a single Markov chain that searches the set for unstable parameters. Our framework leads to a procedure that is capable of performing statistically rigorous tests for instability, which has been extensively tested using
several examples of standard and non-standard queueing networks.
\end{abstract}



\end{frontmatter}

\section{Introduction}
The stability of a Markov chain is arguably among its most important properties. For example, in queueing applications it offers the guarantee that service has been sufficiently provisioned to cope with the load imposed on the network in the long run. For this reason the assessment of the stability of Markov chains has long been an area of intense research.
The objective is often to determine the set of parameter values for which the system's state does not diverge, referred to as the {\it stability region}, of a Markov chain. For many relatively standard Markov chains the stability region is easily expressed in terms of quantities related to the transition probabilities.
However, despite a substantial and growing literature, for a large class of systems determining the stability region has appeared a subtle and highly non-trivial task. Importantly, various (at first sight) counterintuitive results have been found; in particular, for specific queueing models `na\"{\i}vely conjectured'  conditions turn out to be insufficient to ensure stability. 

More specifically, initial results, for example those by Jackson \cite{Ja63}, Baskett {\it et al.} \cite{bcmp75}, and Kelly \cite{ke75}, suggested that the stability of queueing networks would be determined by the network's {\it subcritical region}  (i.e., the set of parameters for which the nominal load at each queue is less than $1$). This conjecture was later proven incorrect by a series of counterexamples that showed instability can occur with subcritical parameters when seemingly benign work conserving rules are applied. Early examples include those of Lu and Kumar \cite{LuKu91}, Rybko and Stolyar \cite{RySt93}, and Kumar and Seidman \cite{KuSe90}. In these examples the instability is typically essentially caused by the priority rules that apply between the customer classes. Similar effects can, however, be constructed in first-in first-out queueing networks with customer classes that have strongly differing mean service requirements, see for example Bramson \cite{Br94}. 
It was thus realized that, at first sight counterintuitively, {\it decreasing} the mean service requirement of certain job classes can in fact induce instability. As a consequence the stability region need not be monotone (nor convex) in its parameters.  For example, Bordenave {\it et al.} provide an instance of a non-convex stability region in \cite{BoMDPr12}.
There are various other examples of queueing networks with unusually shaped stability regions. In \cite{MP07} MacPhee {\it et al.} provide an example with a `thick null recurrent set', in \cite{BaBo99} Baccelli and Bonald show that for certain TCP models the stability region is of a fractal nature, and in  \cite{Nazarathy2015} Nazarathy {\it et al.} investigate a case where the stability region is conjectured to consist of disjoint parts.

To avoid determining the stability regions of queueing networks on a case-by-case basis, various general approaches have been proposed. Perhaps the most straightforward among these amounts to determining the invariant measure of the number of customers; when this allows a normalization, then a stationary distribution exists. This approach works for a set of classical models, relying on concepts such as product form and (quasi-)reversibility \cite{Ke79}, but unfortunately not for many (sometimes just slightly more complex) other systems. 

An arguably more robust approach to determining stability is to construct an appropriate Lyapunov function, and then apply the Foster-Lyapunov theorem. Along these lines Tassiulas and Ephremides, for example, use a quadratic Lyapunov function to find a series of policies which are stable in a wide variety of settings \cite{TaEp91}. Constructing an appropriate Lyapunov function is often specific to the application at hand, but the approach can be simplified by studying the fluid model associated with the queueing network. Such a fluid approach was first described by Rybko and Stolyar \cite{RySt93} and was developed in a general form by Dai \cite{Dai95}; a textbook treatment is presented in Bramson \cite{Br08}. Importantly, for specific models this approach can help determine the conditions under which there is stability, but it does not instantly provide a stability condition for a given network at hand. Thus far, no general framework has been developed that is capable of deciding whether a given Markov chain is stable or not. 

The objective of this paper is to develop a general simulation based approach to determining when a given Markov chain should be classified as unstable. A first paper to consider this approach is \cite{wieland2003queueing}. Then, concurrently to our work, \cite{LeMan} proposes a simulation based method for determining the stability region of a multiclass queueing network with respect to its arrival rate, when it is possible to verify that the stability region satisfies particular stochastic monotonicity properties. Given the variety of models and counter-examples discussed above, we place importance on the generality of settings to which our algorithm is suitable. 
To apply our algorithm we do not place structural assumptions on the stability region, we do not restrict parameters of interest, and we do not restrict the mechanism from which the simulations are derived. We focus on providing theoretical guarantees on the performance of our method for a broad class of models where simulations can exhibit either positive or negative drift.

In particular, rather than specific parameter choices, we are interested in the stability classification of parameter {\it sets}. The distinguishing features are: (i)~that the methodology can be easily used for a relatively broad class of systems, and (ii)~that the technique is based on Monte Carlo simulation. 
Clearly, it is straightforward to develop a simulation-based method that can speculatively test stability for a single parameter setting. It is nevertheless far from obvious how an algorithm should be set up that can identify whether a system is unstable for any of the parameter values within a given set. 

This paper resolves this issue by proposing a {\it simulated annealing} \cite{Kirkpatrick1983} based algorithm that systematically searches the parameter set, and determines whether it contains a subset of positive measure consisting only of unstable parameter values. Instead of having to perform a series of simulations to answer the stability identification question, our algorithm performs a single simulation run of a process that encompasses both the queueing network and the parameter set, which is provably capable of finding positive measurable subsets of unstable parameters.  That is, the output of our algorithm is a statistical statement that provides explicit asymptotic performance guarantees. We view our work as a substantive pioneering study on the {\it simulation based} computation of the stability region of Markov chains. 

The framework we propose has the major advantages over existing ones of being broadly applicable and relying only on mild modeling assumptions. Our method provably provides the correct outcome if the Markov chain has bounded increments. Another significant advantage of the approach is that the annealing algorithm can essentially be performed separately from the simulation of the queueing network; as a consequence, the program can be organized with an inner loop (simulating the queueing network with given parameter values using a rather complex simulator) and an outer loop (simulating the annealing step). It thus enables us to computationally determine the stability region for (i)~relatively straightforward models with non standard features for which this has not been identified in closed form, but also for (ii)~larger, realistic models capturing application-specific details. 

We now proceed by providing an informal description of the setting we consider as well as our algorithm.
The key object in this paper is the collection of Markov chains \[\big((X^{(\lambda)}_k)_{k\ge 0} : \lambda\in\mathcal L\big),\] each of them evolving on the state space $\mathcal X$, where $\mathcal L \subset\mathbb R^I$ is a set of parameter values.
For instance, $\lambda \in \mathcal L$ could parametrize the arrival rates of a queueing network consisting of $I$ queues.
Our algorithm detects if there is a subset $\bar{\mathcal L} \subset \mathcal L$ of positive measure for which the Markov chains  $(X^{(\lambda)} : \lambda\in\bar{\mathcal L})$ are unstable. 
Importantly, for reasons that will become clear, we use a definition of `stability' that differs slightly from those in common use.
We essentially define stability of an individual chain through a Lyapunov drift condition imposed on a function $f$. For example, $f(x)$ could give the total number of jobs in the queueing network when it is in state $x$. Informally speaking, if for $f$ the process $f(X^{(\lambda)})$ has negative drift above some finite threshold, then we call $X^{(\lambda)}$ $f$-stable. Alternatively, if $f(X^{(\lambda)})$ has positive drift above some finite threshold, then we call $X^{(\lambda)}$ $f$-unstable. If there exists a  $\bar{\mathcal L} \subset \mathcal L$ of positive Lebesgue measure such that $X^{(\lambda)}$ is $f$-unstable for all $\lambda \in \bar{\mathcal L}$, then we call $\mathcal L$ $f$-unstable, and otherwise we call $\mathcal L$ stable. 

The main idea behind the algorithm is that it generates a discrete time process with state space $({\mathcal X},\,{\mathcal L})$. Given an initial state $(x,\,\lambda)$,  a new parameter proposal $\gamma$ is chosen uniformly from $\mathcal L$.  The Markov chain $X^{(\gamma)}$ then evolves for $\tau(x)$ time units starting from initial state $x$. In our implementation $\tau(x)$ is chosen to be proportional to $f(x)$. Denoting $X^{(\gamma)}_{\tau(x)} =: y$, the next state of the bivariate process is subsequently chosen by comparing the proposed state $(y,\,\gamma)$ with the current state $(x,\,\lambda)$ according to the Metropolis rule:
\begin{equation}\label{Iteratey}
(x',\,\lambda') = 
\begin{cases}
(y,\,\gamma) & \text{with probability}\quad \exp(\eta\, [f(y) -f(x)]_-),\\
(x,\,\lambda)& \text{otherwise.}
\end{cases}
\end{equation}
Here $[z]_- := \min\{0,\,z\}$ and $\eta$ is a  positive tuning parameter for the algorithm. 

The above iteration is motivated by the simulated annealing algorithm initially proposed by Kirkpatrick, Gelatt, and Vecchi \cite{Kirkpatrick1983}. The main advantage of this type of update is the relative generality of optimization problems that it can provably handle, while still being superior to exhaustive search methods.  A key distinction between our method and the typical implementation of simulated annealing is that our cooling schedule is achieved using a combination of the fixed parameter $\eta$ and the parameter $\tau(x)$, that varies with the state of the Markov chain. 

In addition to the global search algorithm just described, we will also study a local search version. This version is different in two respects. Firstly, the new parameter proposal $\gamma$ is sampled uniformly from the neighborhood of the current parameter, a set we denote by $\mathcal{B}_\lambda$. Secondly, we allow $X^{(\lambda)}$ and $X^{(\gamma)}$ to evolve for $\tau(x)$ time units starting from initial state $x$, let $X^{(\lambda)}_{\tau(x)} := y'$ and then apply the above Metropolis rule with $x$ replaced by $y'$. Our key theorems apply to both versions of the algorithm and we explore differences in performance of the two methods through examples.

After having pointed out how the algorithm works, we now provide results that separate its sample paths into either the stable or unstable regimes. 
Let $S_k=(Y_k,\,\Lambda_k)$ be the state of the bivariate process achieved after $k$ iterations of the above rule \eqref{Iteratey} and let $T_k$ be the total amount of time that the algorithm has run for by the $k$th iteration.
The first main theoretical contribution of this paper, later stated formally in Theorem~\ref{thm:stable}, shows under mild conditions on $X$, that if there does {\it not} exist a subset of unstable parameters $\bar{\mathcal L} \subset \mathcal L$ with positive Lebesgue measure, then almost surely
\begin{equation}\label{Bd2}
\lim_{k\rightarrow\infty} \frac{f(Y_k)}{T_k} = 0\,.
\end{equation}
Importantly, the second main theoretical contribution of this paper, later stated formally in Theorem~\ref{thm:unstable}, shows that there exists a true `dichotomy' since if there does exist an unstable set of parameters $\bar{\mathcal L} \subset \mathcal L$ with positive Lebesgue measure, then the process $(Y_k : k\in\mathbb Z_+)$ diverges, in the sense that, almost surely
\begin{equation}\label{Bd1}
\liminf_{k\rightarrow \infty} \frac{f(Y_k)}{T_k} > 0\,.
\end{equation}

The bound \eqref{Bd1} relies on a martingale argument in combination with an application of the Azuma--Hoeffding inequality. The bound \eqref{Bd2} is proven by a coupling argument: as it turns out, in the stable situation the process $f(Y)$ can be majorized by a Markov chain $(W_k: k\in\mathbb Z_+)$ that has an asymptotic drift of zero. An advantage of the coupling approach used to prove \eqref{Bd2} is that the Markov process $W$ is easily simulated, and can therefore be used to provide probabilistic bounds on the likelihood of instability. We use this approach to perform rigorous statistical tests for instability of the underlying set $\mathcal L$. There are many potential approaches to the adaptation of our theoretical results to a practical test for instability.  The approach we suggest is to compare $f(Y_k)$ with the quantiles of $W_k$. Specifically, we let the process $f(Y)$ evolve according to the rule given in \eqref{Iteratey} until the total number of steps in $\mathcal X$ taken between steps of $Y$ exceeds some predetermined level $k^*$, at which point we compare $f(Y_k)$ with the $1-\alpha$ quantile of $W_k$, with $\alpha$ being the desired confidence level. If $f(Y_k)$ exceeds this quantile then we obtain a strong rigorous statistical statement of instability, whereas otherwise we fail to reject the `null hypothesis' of stability. 

We provide five example applications of our algorithm. We apply it to a system consisting of a set of parallel queues studied by Tassiulas and Ephremides in \cite{Tassiulas1993}, a tandem queueing system, the celebrated Rybko--Stolyar network \cite{RySt93}, a network of input queued switches studied by Andrews and Zhang in \cite{AndrewsZhang2003}, and a broken diamond random access network (RAN) recently studied by Ghaderi et al.\ in \cite{Ghaderi2014}. 

Since the stability region is well known for the parallel and tandem systems, these are ideal examples on which to verify that the algorithm performs as expected. We use the tandem system to show that although our results are in a discrete time setting, we are still able to effectively study continuous time systems using a jump chain associated with the process. Additionally, this network also highlights that we are able to test multidimensional parameter sets for instability, and suggests that we are able to relax the Markov assumption. The Rybko--Stolyar network is a popular example of a system with oscillating queue sizes.  Not only does our analysis confirm existing theoretical results that give sufficient conditions for stability of this system, it also provides a statistically rigorous statement that these conditions are also necessary. The network of input queued switches allows us to show that our algorithm provides interesting results for systems with high dimensional state spaces and complex dynamics. In addition, we are able to use our methodology to show that this is an example of a system where the longest queue first policy is not maximally stable. Our final example, the RAN of Ghaderi et al., is currently a hot topic of research in the applied probability community. This system exhibits oscillatory queue size sample path behavior reminiscent of the Rybko--Stolyar network, but in a higher dimensional setting. We are able to expand on the theoretical results of \cite{Ghaderi2014} by providing more specific (statistical) information about which parameter sets are unstable. Throughout this section results are given in terms of both the global and local versions of the algorithm. For some of the models (parallel, tandem, Rybko--Stolyar), the local algorithm appears to perform better, while for others (switches, RAN) the global algorithm appears to be superior.

The remainder of this paper is structured as follows. In Section~\ref{Sec2} we give a formal description of our framework and the assumptions imposed.  Section~\ref{Sec3} presents the algorithm and states our main results, i.e., Theorem~\ref{thm:unstable} and Theorem~\ref{thm:stable}.  In Section~\ref{Proofs} detailed proofs are given (of our main results, propositions, and lemmas). We then provide a range of case studies in Section~\ref{Cases} that demonstrate the algorithm's potential. Section~\ref{Disc} presents concluding remarks as well as an outlook on future research. 

\section{Framework}\label{Sec2}
In this section we present the setup considered in the paper.
The object of study is the irreducible Markov chain $X^{(\lambda)}$ that is parametrized by parameter $\lambda\in\mathcal L$; these parameters can, for example, be thought of as the arrival or service rates in a queueing network. It is assumed throughout that ${\mathcal L}$ is a closed subset of ${\mathbb R}_+^I$, for some $I\in{\mathbb N}$, with finite positive Lebesgue measure (which is denoted $|{\mathcal L}|$). The Markov chain, which may represent the evolution of the population of a queueing network, attains values in ${\mathcal X}:={\mathbb Z}_+^J := \{0, 1, \dots\}^J$, for some $J\in{\mathbb N}$. 

As pointed out in the introduction, the main goal of this paper is to devise a procedure that identifies if a parameter set contains any unstable parameters. Put more precisely, the algorithm verifies  whether or not there is a subset $\bar{\mathcal L}$ of ${\mathcal L}$ such that for all $\lambda \in{\bar{\mathcal L}}$  the associated Markov chain is \textit{unstable}.

Further, for each $\lambda\in \mathcal L$, we let ${\mathcal B}_\lambda$ be the neighborhood of $\lambda$. As is commonly assumed in local search algorithms, we assume that $\lambda\in \mathcal B_\lambda$ and for any $\lambda_1, \lambda_n\in{\mathcal L}$ there is a sequence of neighborhoods with $\lambda_{k+1}\in {\mathcal B}_{k}$ for $k=1,...,n-1$. 

We will work extensively with a Lyapunov function that maps the state of the Markov chain to a nonnegative real number, that is a monotone function $f:{\mathcal X}\rightarrow [0, \infty)$. In our queueing network example, $f(x)$ could represent the \textit{sum} of the queue sizes within the network (that is, the total network population). 
We assume that $f$ is unbounded in the sense that
\[
\liminf_{|x|\rightarrow \infty} f(x) = \infty\,.
\]
It is assumed throughout that for all $\lambda$ the process $f(X^{(\lambda)})$ has bounded increments, implying there exists a constant $\phi_{f} > 0$ such that
\begin{equation}\label{phif}
\left| f(X^{(\lambda)}_{k+1}) - f(X^{(\lambda)}_k)\right| \leq \phi_{f}\,. 
\end{equation} 

We now provide the formal definitions of stability and instability, as used in this paper. 
\begin{definition}\label{StableDef}
Given $f$, we say that the set of parameters $\mathcal{L}$ is \emph{$f$-stable} if  there exists $\delta>0$, $\sigma >0$, and $\kappa>0$ such that
\begin{equation}\label{stable}
{\mathbb E}\left[\left.f(X^{(\lambda)}_k)-f(X^{(\lambda)}_0)\,\right|\,X^{(\lambda)}_0=x\right]\leq -\delta\,\sigma
\end{equation}
for all $x$ such that $|x|\geq \kappa$, for all $\lambda \in \mathcal{L}$, and all $k\ge \sigma$.

Similarly, we say that the set of parameters $\mathcal{L}$ is \emph{$f$-unstable} if  there exists a set $\bar{\mathcal L} \subset {\mathcal L}$ of positive measure, $\delta>0$, $\sigma>0$, and $\kappa>0$ such that
\begin{equation}\label{unstable}
{\mathbb E}\left[\left.f(X^{(\lambda)}_k)-f(X^{(\lambda)}_0)\,\right|\,X^{(\lambda)}_0=x\right]\ge \delta\,\sigma
\end{equation}
for all $x$ such that $|x|\geq \kappa$, for all $\lambda \in \bar{\mathcal{L}}$, and all $k\ge \sigma$.
\end{definition}
For a given value of $\lambda$, the conditions \eqref{stable} and \eqref{unstable} are Lyapunov conditions for which one can obtain positive recurrence or transience of the Markov chain $X^{(\lambda)}$, respectively (see for instance \cite{Hairer2010}).  
We remark that a countable state space Markov chain is positive recurrent if and only if there exists a Lyapunov function $f$ for which is $f$-stable, see Meyn and Tweedie \cite[Theorem 11.0.1]{meyn2012markov}.  In our definition of $f$-stable we consider over \emph{set} of Markov chains for which the same choice of $f$ provides positive recurrent.
In this sense, the definitions of `stable' and `unstable' then ask whether or not the Markov chains $X^{(\lambda)}$ are positive recurrent for parameters $\lambda$ in ${\mathcal L}$. 

We further remark that if a fluid limit, $\bar{f}(\bar{X}^{(\lambda)})$, exists for each (rescaled) process, 
\[
\left(\frac{f(X^{(\lambda)}_{\lfloor kt\rfloor})}{k} : t\geq 0\right)\,,\qquad\lambda \in \mathcal L\,,
\]
then the above conditions \eqref{stable} and \eqref{unstable} imply that 
\[
\frac{\d \bar{f}(\bar{X}^{(\lambda)}(t))}{\d t} \leq -\delta \qquad\text{and}\qquad \frac{\d \bar{f}(\bar{X}^{(\lambda)}(t))}{\d t} \geq \delta 
\]
for $\bar{X}^{(\lambda)}(t) > 0$. In other words, \eqref{stable} and \eqref{unstable} respectively imply fluid stability and fluid instability (see for instance \cite{Br94}). In general, fluid stability and instability are not equivalent to the positive recurrence and transience of an underlying Markov process. Nevertheless, the Lyapunov analysis of fluid models remains one of the most widely deployed and established devices used to determine the positive recurrence and transience of Markov processes. Similarly, our work provides a broadly applicable technique that may be used to determine the positive recurrence and transience of families of Markov processes.

Now that we have introduced our framework, the next section describes our algorithm, as well as the main results upon which the algorithm is based.

\section{Implementation and main results}\label{Sec3}
In this section we explicitly give our algorithm and provide a detailed discussion of the choices underlying it. We then give in Theorem~\ref{thm:stable} and Theorem~\ref{thm:unstable} our main theoretical contribution. We follow this up with a suggested method of using our results to implement actual tests for instability. 

\subsection{Algorithm}
We now describe our algorithm for identifying whether a parameter set is unstable. 
Our approach is based on the principle of searching the relevant parameter set for a parameter choice that maximizes the drift of the Markov process under consideration. As such, well known optimization algorithms provide an ideal source of inspiration for potential methods. As previously mentioned, the approach taken in this paper is based on the well known simulated annealing optimization algorithm. While many other optimization techniques have found acceptance through testing against well known `hard' problems, the simulated annealing algorithm has shown itself to be amenable to rigorous results on performance guarantees. 

In this section we provide a detailed description of our algorithm, and through the use of two theorems provide guarantees on its asymptotic performance. A key advantage of this strong theoretical grounding is that the machinery used to provide these theoretical guarantees also allows us to develop a hypothesis test that outputs a statistical statement of whether or not a Markov chain is stable given a particular parameter set. In this section we also include an illustration of the algorithm and its output in the context of a single server discrete time queueing system. In later sections we demonstrate the algorithm's potential through a series of experiments concerning more complex systems. 

We assume that $\tau(x)=c\,f(x)+d$, where $c,\,d \in (0, \infty)$ are chosen by the algorithm's user. Note that this implies $\tau(x)\rightarrow \infty$ as $|x|\rightarrow \infty$ and that $\tau$ has bounded increments. Finally, let $T_k$ give the time that our chain has been running for at the $k$-th step, that is,
\[
T_k = \sum_{i=0}^{k-1} \tau(Y_i)\,.
\]

We now have all of the machinery needed to give both versions of our algorithm. 
\begin{algorithm}\label{ALG1}
Global search algorithm: \\
{\tt
Initialize: Set $k =1$, $T_0 = 0$, choose $Y_0$ from $\mathcal X$, and $\Lambda_0$ from $\mathcal L$.
\begin{enumerate}[(i)]
\item For $x = Y_{k-1}$ and $\tau = \tau(Y_{k-1})$, set  $T_k = T_{k-1} + \tau$ and sample $\gamma \sim {\sf Uniform}(\mathcal L)$.
\item Sample $y=X^{(\gamma)}_{\tau}$ conditional on $X^{(\gamma)}_0 = x$.
\item For  $\lambda = \Lambda_{k-1}$, set
\begin{align}\label{ChangeRule2}
(Y_k\,,\Lambda_k) = 
\begin{cases}
(y,\,\gamma) & \text{with probability}\quad \e^{\eta\, [f(y) -f(x)]_-},\\
(x,\,\lambda)& \text{otherwise.}
\end{cases}
\end{align}
\item If \emph{stopping condition} is met, then stop, else set $k=k+1$ and return to (i). 
\end{enumerate}}
\end{algorithm}
As outlined in the introduction, each step of the global search algorithm compares $x$, as sampled in the previous step, with a new value $y$, sampled using a uniformly at random selected parameter $\gamma$ from $\mathcal L$ with runtime $\tau(x)$ and initial state $x$. The state is then updated according to the Metropolis rule \eqref{ChangeRule2}.

Recalling from Section~\ref{Sec2} that $\mathcal B_\lambda$ is a neighborhood of $\lambda$ in $\mathcal L$,  the local search version operates as follows.

\begin{algorithm}\label{ALG3}
Local search algorithm: \\
{\tt
Initialize: Set $k =1$, $T_0 = 0$, choose $Y_0$ from $\mathcal X$, and $\Lambda_0$ from $\mathcal L$.
\begin{enumerate}[(i)]
\item For $x = Y_{k-1}$, $\lambda = \Lambda_{k-1}$ and $\tau= \tau(Y_{k-1})$, set  $T_k = T_{k-1} + \tau $ and sample $\gamma \sim {\sf Uniform}(\mathcal B_\lambda)$.
\item Sample  $x' = X^{(\lambda)}_{\tau}$ conditional on $X^{(\lambda)}_0 = x$.
\item Sample $y = X^{(\gamma)}_{\tau}$ conditional on $X^{(\gamma)}_0 = x$.
\item Set
\begin{align}\label{LocalChangeRule}
(Y_k\,,\Lambda_k) = 
\begin{cases}
(y,\,\gamma) & \text{with probability}\quad \e^{\eta\, [f(y) -f(x')]_-},\\
(x',\,\lambda)& \text{otherwise.}
\end{cases}
\end{align}
\item If \emph{stopping condition} is met, then stop, else set $k=k+1$ and return to (i). 
\end{enumerate}}
\end{algorithm}
The local search algorithm compares states $(x',\,\lambda)$ and $(y,\,\gamma)$ where $x'$ is sampled by running the current parameter $\lambda$ for a further $\tau$ steps and $(y,\,\gamma)$ is sampled by running a neighboring parameter $\gamma \in \mathcal B _\lambda$ for the same number of steps. These states are then compared according to the Metropolis rule \eqref{LocalChangeRule}.

As we will discuss in more detail, the relative performance of the global search and local search differ depending on the model and setting to which they apply. Under general modeling assumptions both algorithms converge to a behavior that only accept unstable parameters in $\mathcal L$. The Global Search Algorithm proposes parameters uniformly at random and thus asymptotically will only accept parameters uniformly at random in the unstable set $\bar{\mathcal L}$. This is useful if one wants to identify the region of instability, in addition to determining if instability occurs. The Local Search Algorithm proposes two neighboring parameters and compares them simultaneously. In this way the Local Search Algorithm applies a hill-climbing heuristic. In this sense it is more aggressive in approaching regions of instability, but will not identify the entire unstable region.

When analyzing Algorithm~\ref{ALG1}, we assume that $\mathcal L$ is a general measurable set and that $\bar{ \mathcal L}$ is a set with positive Lebesgue measure, while for the local search Algorithm~\ref{ALG3} we place some restrictions. We assume the following:
\begin{assumption}\label{LocalAssump}
When analyzing the local search algorithm we assume that $\mathcal L$ is a finite set where for each $\mathcal L ' \subset \mathcal L$ either $\mathcal L'$ is unstable or $\mathcal L'$ is stable, according to Definition~\ref{StableDef}. Further, we assume that there exists a state $x_0$ where
\begin{equation}\label{Reduce}
 \mathbb P ( X_1^{( \lambda)} =x_0\,|\,X_0^{( \lambda)} =x_0 ) > 0\,.
\end{equation}
\end{assumption}
In a queueing setting $x_0$ may, for example, correspond to a state where all queues are idle.

An important feature of our work is that we do not place structural conditions on $\mathcal L$ such as convexity or monotonicity. Since examples of Markov processes violating these conditions frequently occur in both theory and practice, by avoiding such conditions our work is widely applicable. Another key feature is that we do not assume knowledge of the process generating each sample path is available, we only require samples of the state description in response to parameter choices. This further extends the set of models that may be analyzed using our method, since practical simulators (although Markovian) are often not generated from a simple closed form Markovian descriptor (transition matrix or infinitesimal generator), but rather come in the form of a `black box' that provides outputs in response to parameter inputs. 

Note that we have not provided an explicit stopping condition for the algorithm yet. Since our results are asymptotic in the number of steps $k$, it may be sensible to run the algorithm until some large $k$, chosen based on CPU time limitations. An alternative may be to dictate a particular total budget of time that the algorithm may evolve in $\mathcal X$. To do this, choose a $k^*$ and run the algorithm until $T_k > k^*$. In either case it may not be obvious whether the sample path belongs to the stable or unstable regimes, an issue that we address with a test for instability in Section~\ref{Test}. 

We now briefly address some of the choices we made in the design of the algorithm. 
Firstly, note that the $\tau$ function we introduced has replaced the cooling schedule from the traditional simulated annealing algorithm. The functional form of $\tau$ ensures that when $Y_k$ is large the subsequent $\Lambda_{k+1}$ proposal is given an increased opportunity to demonstrate that it has higher drift. Since a large $Y_k$ hints that an unstable parameter choice has been recently chosen this helps to ensure that CPU budget is expended comparing parameter choices which appear to be unstable. 

The conditioning in step (ii) of the algorithm on $X^{(\gamma)}_0 = x$, rather than starting each new sample from $X^{(\gamma)}_0$ equal to zero, is intentionally designed to allow the system to build up to a size where instability properties become evident. That is, as per Definition~\ref{unstable}, the drift properties we are seeking only become evident after $|x| > \kappa$ has occurred. Forcing the system to reach $|x| > \kappa$ in a single $X^{(\gamma)}_{\tau(x)}$ sample may result in the algorithm inefficiently repeating `burn in' time. 

\subsection{Main results}
Our main theoretical contributions are Theorem \ref{thm:unstable} and Theorem \ref{thm:stable} below. These demonstrate the stability of a set $\mathcal L$ can be summarized asymptotic sample path behavior of the process $f(Y)/T$. In essence the stability of a parameterized family of Markov processes can be summarized by the stability of a single Markov process, as generated by Algorithm~\ref{ALG1} or Algorithm~\ref{ALG3}. 

The main results of this paper are as follows:
\begin{theorem} \label{thm:stable}
If the set $\mathcal{L}$ is stable then, almost surely,
\begin{equation} 
\lim_{k\to\infty}\frac{f(Y_{k})}{T_k}=0\,.\label{liminfFstable}
\end{equation}\end{theorem}
 Theorem~\ref{thm:stable} shows that when $\mathcal{L}$ is stable, the sample path of $f(Y)/T$ converges to 0. 
 \begin{theorem} \label{thm:unstable}
If the set $\mathcal{L}$ is unstable then, almost surely,
\begin{equation} 
\liminf_{k\to\infty}\frac{f(Y_{k})}{T_k}>0\,.\label{liminfFunstable}
\end{equation}
\end{theorem}
Theorem~\ref{thm:unstable} shows that when $\mathcal{L}$ is unstable, the sample path of $f(Y)/T$ eventually never returns to 0. In practical use it is $f(Y_k)/T_k$, for some large $k$, that is observed, rather than its limiting value. It is therefore not possible to directly apply the theorems. Instead, when $f(Y)/T$ appears to converge to 0, the contrapositive of Theorem~\ref{thm:unstable} provides evidence that the parameter set is not unstable. Conversely, when $f(Y)/T$ appears to diverge, converge to a positive constant, or fluctuate within a set that does not contain 0, the contrapositive of Theorem~\ref{thm:stable} provides evidence that the parameter set is not stable. 

We now include a short example to illustrate these theorems. Consider a simple discrete time queueing system where an arrival occurs at the beginning of each time slot with probability $p \in [0,\,1]$, and then subsequently, if the queue is non-empty, a service occurs with probability 0.5. Clearly, so long as the queue is non-empty the expected change in queue size between time periods is $p-0.5$. Hence, for $p<0.5$ the system is $L^1$-stable with $\kappa = \sigma = 1$ and $\delta = 0.5-p$. Figure~\ref{fig:SimpleSamplePaths} illustrates Theorem~\ref{thm:unstable} and Theorem~\ref{thm:stable} using the sample path behavior of $f(Y)/T$ for this simple system. The sample paths corresponding to $p$ sampled from $\mathcal L = [0,\,0.4]$ appear to converge towards 0, providing evidence that this set is not unstable. Similarly, the sample path corresponding to $p$ sampled from $\mathcal L = [0,\, 0.6]$ appears to remain constant at approximately $10^{-2}$ in the global case and appears to diverge in the local case, providing evidence that this set is not stable. 

We remark that the behavior of the unstable sample path substantially differs between the global and local versions of the algorithm. In the global case the unstable sample path quickly separates from the stable sample path and appears to tend towards some constant value. For the local algorithm, however, the stable and unstable sample paths appear highly similar until suddenly the unstable sample path rapidly increases. This suggests that if $n$ is not large enough, the local algorithm may perform very poorly, however for $n$ large it may perform vastly better. 

\begin{figure}[h]
	\centering
	\includegraphics{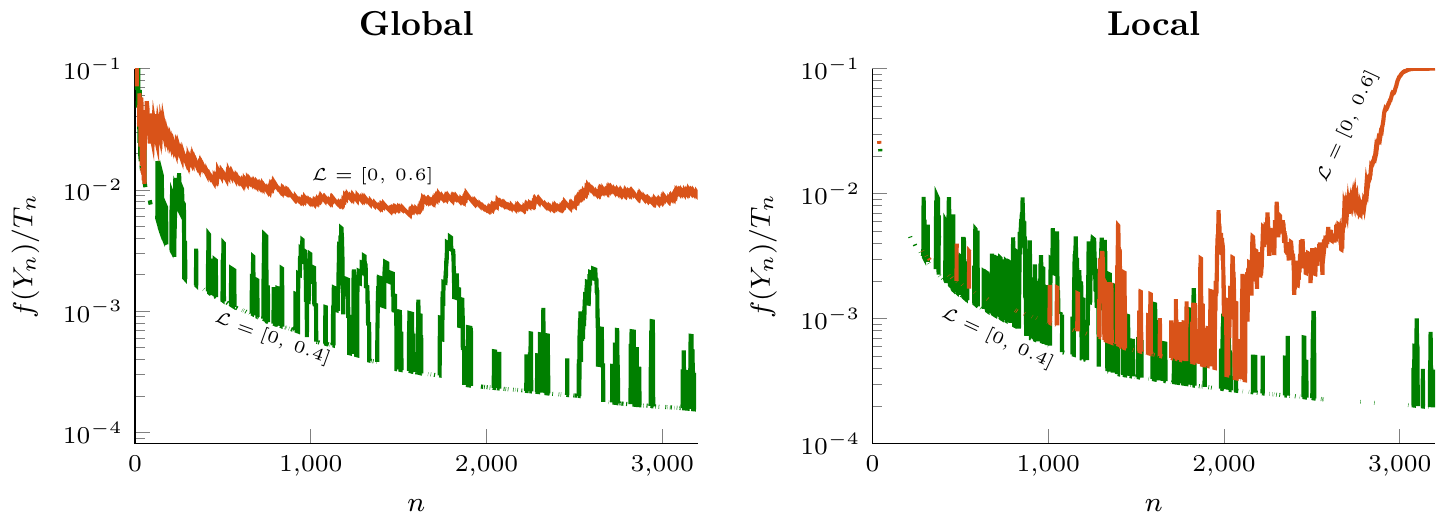}
	\caption{Comparison of $f(Y)/T$ sample paths for stable ($\mathcal L = [0,\,0.4]$) and unstable ($\mathcal L = [0,\,0.6]$) parameter sets when the global and local versions of the algorithm are applied to a simple queue.}
	\label{fig:SimpleSamplePaths}
\end{figure}

It is not necessarily clear in finite time if a sample path of $f(Y)/T$ belongs to the regime of Theorem~\ref{thm:stable} or Theorem~\ref{thm:unstable}. In the next subsection we address this issue by presenting a test for instability. 

\subsection{A test for instability}\label{Test}
We  now provide a method to test, statistically, whether or not a parameter set is unstable. 
Here one could consider a null-hypothesis which states that the parameter set is stable for some given $\delta$, cf. \eqref{stable}. 
Given this and the simulated model, we can construct a closed form family of random variables $Z(w)$, $w\geq 0$ (given by Lemma~\ref{hoefd} in Section~\ref{STAB}) such that $Z(w)$ stochastically majorizes the increments of $f(Y_k)$. 
With this choice of $Z(w)$, we can then define the Markov chain $(W_k : k\in\mathbb Z_+)$ according to the recursion
\begin{equation}\label{defW}
W_k = W_{k-1} + Z(W_{k-1})\,.
\end{equation}
The following proposition will be proven to show that there is a coupling where the Markov chain $(W_k : k\in\mathbb Z_+)$ stochastically dominates $(f(Y_k): k\in\mathbb Z_+)$. 

\begin{proposition}\label{Coupling}
For stable $\mathcal L$, when $f(Y_0) \leq W_0$, there exists a coupling between $(Y_k : k\in\mathbb Z_+)$ and $(W_k : k\in\mathbb Z_+)$ such that
\[
f(Y_k) \leq W_k,\qquad \text{for all } k.
\] 
\end{proposition}

Since Theorem~\ref{thm:unstable} says that $f(Y)$ will diverge in the unstable case, we suggest comparing $f(Y_k)$, as outputted by Algorithm~\ref{ALG1}, with the quantiles of $W_k$. In particular let, $q^{(\alpha)}_k$ be such that $\P(W_k > q^{(\alpha)}_k) = 1-\alpha$. Note that, given a problem instance chosen according to Definition \ref{StableDef} and \eqref{phif} (that is, particular values of $\phi_f$, $\delta$, $\sigma$, and $\kappa$), the quantiles of $q^{(\alpha)}$ can be estimated quickly and easily through Monte Carlo simulations of the $W$ process.
 If $f(Y_k) > q^{(\alpha)}_k$ then we suggest concluding that the parameter set is $f$-unstable for that problem instance. Otherwise we suggest that there is not enough evidence to make a conclusion either way. 

To illustrate this approach we return to the simple example introduced in the previous section. In Figure~\ref{fig:SimpleHypTest} an estimated $q^{(0.05)}$ curve with $\delta = 0.05$, $\tau(x) = 0.5\,x + 1$, and $\sigma = \kappa = Y_0 = 1$ is compared with estimated mean curves for the $f(Y)$ process with $\mathcal L = [0, \ell]$ for $\ell = 0.6,\,0.55,\,0.5,\,0.45,$ and $0.4$. As expected the $q^{(0.05)}$ curves bound the mean curves of $f(Y)$ for $\ell < 0.5$, while for $\ell > 0.5$ the mean curves of $f(Y)$ appear to eventually exceed the $q^{(0.05)}$ curve. With reference to Definition~\ref{StableDef}, $\delta$ is the downward drift that a process must exhibit  in order to be stable. As can be seen here, it is to be expected that $q^{(0.05)}$ curves generated using a particular $\delta$ value bound those generated using higher values of $\delta$.  Since we test for a stabilizing drift up to $\delta$,  it is desirable to use a $\delta$ which is as low as possible. In Section~\ref{Cases} we investigate further the trade-off between $k^*$ and $\delta$ that users of our algorithm must keep in mind. Note that the global and local $q^{(0.05)}$ curves are nearly indistinguishable from each other here. 

\begin{figure}[h]
	\centering
	\includegraphics{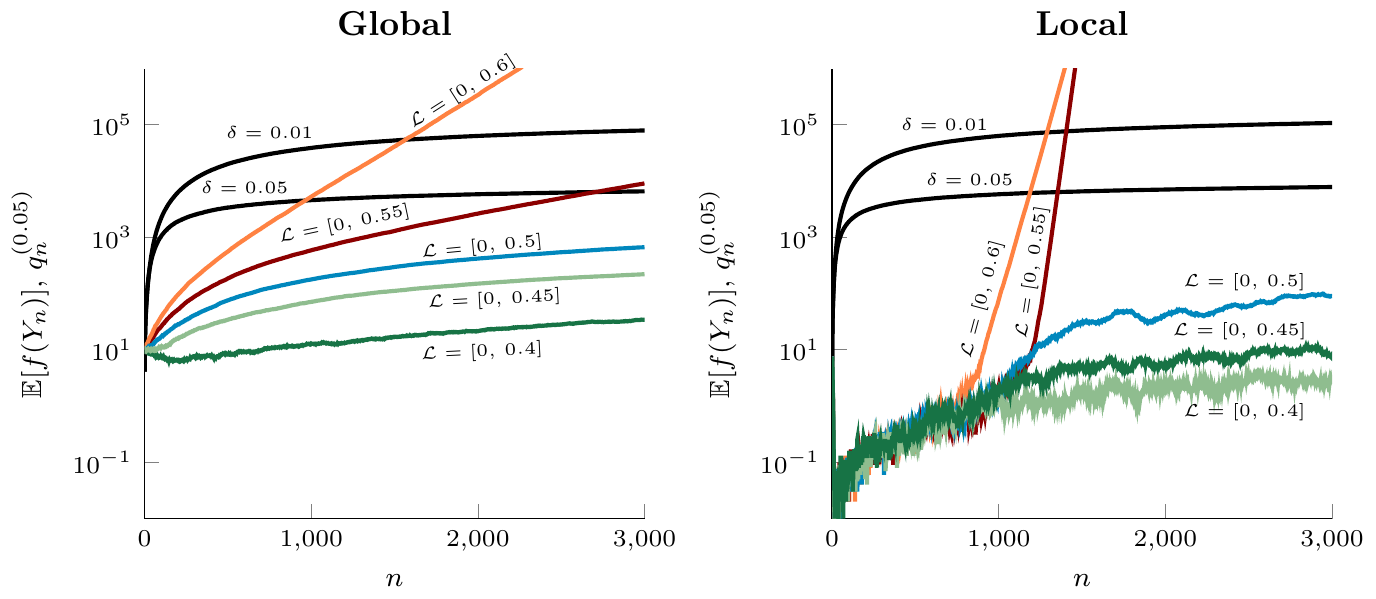}
	\caption{ Estimated mean curves of $f(Y)$ when $\mathcal L = [0, 0.4],$ $[0, 0.45],$ $ [0, 0.5],$ $ [0, 0.55],$ or $[0, 0.6]$ for a simple queue compared to estimated $q^{(0.05)}$ curves with $\delta = 0.01$ or $\delta = 0.05$.}
	\label{fig:SimpleHypTest}
\end{figure}

In some instances, obtaining long sample paths of $f(Y)$ may be a computationally intensive task.  We now describe an approach to managing the user's simulation budget, but various other approaches could be taken. In order to achieve a significance level of \emph{at least} $\alpha$, we propose choosing a simulation budget $k^*$, to take the first sample point $f(Y_k)$ such that $T_{k+1} > k^*$ and to then compare this $f(Y_k)$ with an estimate of $q^{(\alpha)}_k$. If $f(Y_k)$ exceeds $q^{(\alpha)}_k$ then we suggest rejecting the `null hypothesis' of stability, and otherwise we suggest concluding that there is not enough evidence to make a conclusion. This is the approach that we take in Section~\ref{Cases}. Note that this does not involve comparing the `test statistic' $f(Y_k)$ with its distribution, but rather we compare it with a distribution which is stochastically dominant. Assuming that the $1-\alpha$ quantile estimate for $W_k$ is accurate, asymptotically in $k^*$ the significance level will in fact tend to 0 for all $\alpha>0$, and never exceed $\alpha$.  

We summarize the above discussion in Algorithm~\ref{ALG2}, below. Note that this algorithm is just one of many potential extensions of our basic Algorithm~\ref{ALG1}, for which we give specific theoretical results, and an input to this algorithm is an appropriate $q^{(\alpha)}$ estimate. 

\begin{algorithm}\label{ALG2}
Stability test algorithm: \\
{\tt
Initialize: Set $k =1$, $T_0 = 0$, choose $Y_0$ from $\mathcal X$, and $\Lambda_0$ from $\mathcal L$.
\begin{enumerate}[(i)]
\item For $x = Y_{k-1}$ and $\tau = \tau(Y_{k-1})$, set  $T_k = T_{k-1} + \tau$ and sample $\gamma \sim {\sf Uniform}(\mathcal L)$.
\item Sample $y=X^{(\gamma)}_{\tau}$ conditional on $X^{(\gamma)}_0 = x$.
\item For  $\lambda = \Lambda_{k-1}$, set
\begin{align}\label{ChangeRule3}
(Y_k\,,\Lambda_k) = 
\begin{cases}
(y,\,\gamma) & \text{with probability}\quad \e^{\eta\, [f(y) -f(x)]_-},\\
(x,\,\lambda)& \text{otherwise.}
\end{cases}
\end{align}
\item If $T_k+\tau(Y_k) > k^*$, then proceed to (v), else set $k=k+1$ and return to (i). 
\item If $f(Y_k) > q_k^{(\alpha)}$, then conclude $\mathcal L$ is $f$-unstable, else repeat from (i). 
\end{enumerate}}
\end{algorithm}

In Section \ref{Proofs} we prove the results presented above, and then in the Section \ref{Cases} we will demonstrate the algorithm's potential on some more complex systems. 

\section{Proofs}\label{Proofs}
We first prove Theorem~\ref{thm:stable} in Section~\ref{STAB}, which applies to the stable regime, in the context of the global search algorithm and provide a remark on the minor modifications to this proof that would be needed to show the local search case. We then prove Theorem~\ref{thm:unstable}, which applies to the unstable case, for the global search and local search algorithms in Section~\ref{UNS} and Section~\ref{Sa} respectively. 

\subsection{Stable parameter set}\label{STAB}
In this subsection we prove Theorem \ref{thm:stable}. In what follows, we first give a formal definition of the random variables $Z(w)$. Then, to prove Theorem \ref{thm:stable}, we require Proposition~\ref{Coupling}, given in Section~\ref{Sec3}, and Proposition~\ref{WZero}, given below. The first of these propositions shows the existence of a process that majorizes any $Y$ process generated from a stable parameter set. The second proposition then shows that this majorizing process has an asymptotic drift of zero, which leads to the result of the theorem. In order to obtain Proposition~\ref{Coupling} we require Lemma~\ref{hoefd}, given next, and Lemma~\ref{Zmonotone}, which is a simple technical lemma that can be found in the appendix. Lemma~\ref{hoefd} explicitly provides a level dependent random variable that bounds the jumps of the $f(Y)$ process and Lemma~\ref{Zmonotone} gives a useful monotonicity property for these jumps. Proposition~\ref{WZero} is proven by contradiction and depends on Lemma~\ref{LemmaZ}  which states that the sequence of random variables defined in Lemma~\ref{hoefd} are square integrable and tend to an expectation of zero as the level diverges. 

In this section all of the proofs are performed in the context of the global search algorithm, however at the end of the section we remark on the minor modification required to adapt the proof to the local search algorithm context. \\

We develop a process $W$ that stochastically majorizes any $f(Y)$ process generated from a stable parameter set. Recall $\phi_f$, $\delta$ and $\sigma$ from Definition \ref{StableDef}. We define the function $n$ such that $n(w)$ is the smallest integer such that $\sigma\,n(w) \ge \tau(w)$ when the underlying process is in a state $f^{-1}(w)$.  We bound the jumps of $f(X^{(\lambda)})$, for a given \emph{stable} $\lambda$. The following lemma provides this bound.  

\begin{lemma}\label{hoefd}
If $\lambda$ is $f$-stable, then there exists random variables $(Z(w) : w \ge 0 )$ and a constant $w^*$ such that, for all $x$ with $f(x)\ge w^*$,
\[
{\mathbb P} \left( f(X^{(\lambda)}_{\tau(x)}) - f(x) \geq z\,\Big|\,X^{(\lambda)}_0=x \right)
\leq 
{\mathbb P} (Z(f(x)) \ge z)\,,
\]
where, for $\sigma$, $\kappa$ and $\delta$ as given in \eqref{stable}, $Z(w)$ is a random variable with distribution
\begin{align}
&{\mathbb P}(Z(w)\ge z) =\notag\\ 
&\begin{cases}
1\wedge \Big[ \exp\Big(- {\displaystyle \frac{(z-\alpha_1(w))^2}{2\alpha_2(w)}}\Big) + n(w)\,\exp\Big(- {\displaystyle \frac{(z -\alpha_3(w))^2}{2 \alpha_4(w)}} \Big)\Big]\,, & \text{if } z > 0,\\
1\,, & \text{otherwise.}
\end{cases} \label{NastyBound}
\end{align}
where 
\begin{eqnarray*}
\alpha_1(w)&=&\sigma\phi - \sigma n(w) \delta\,, \qquad \alpha_2(w)=(\phi_f+\delta)^2 \sigma^2 n(w)\,, \\
\alpha_3(w)&=&\sigma\phi - w +\kappa\,,\qquad\;\; \alpha_4(w)=\phi_f^2 \,\sigma^2 n(w)\,.
\end{eqnarray*}
\end{lemma}

The proof of this lemma is straightforward, yet, somewhat technical; a proof is given in the Appendix. The form of the expression for ${\mathbb P} (Z(w) \ge z)$ given above can be understood as follows. By \eqref{stable} the stable downward drift condition only applies when the chain has run for at least $\sigma$ steps, so we consider the process on steps of size $\sigma$ and ensure that we take enough of these steps, $n(w)$, to exceed the $t$ of interest. The maximum is a result of the trivial upper bound on probabilities, and the $1+n(w)$ exponential terms correspond to a union bound using an equivalent number of applications of the Azuma--Hoeffding inequality.

The downward drift condition requires $|x| > \kappa$, and so we apply Azuma--Hoeffding to different martingales depending on whether the sample path of interest enters the states $\{x : |x| < \kappa\}$ or not. The first exponential term corresponds to sample paths that never enter $|x| < \kappa$, and so the martingale we use does not include $\kappa$ and has steps which are bounded by $(\phi_f+\delta)\,\sigma$. The remaining exponential terms correspond to sample paths that hit the level $\kappa$. We associate with such sample paths a martingale which reflects the fact that there be must be an excursion from $\kappa$ to $z+f(x)$ that can be stopped just before this excursion occurs. As such these remaining exponentials do not rely on $\delta$.

From Lemma \ref{hoefd} we can prove Proposition \ref{Coupling}.

\begin{proof}[Proof of Proposition~\ref{Coupling}]
The inequality in Lemma \ref{hoefd} bounds the upward movement of the Markov process $X_\lambda$. 
For $w_0 = f(x_0)$, we see that
\begin{equation}\label{BoundYbyZ}
{\mathbb P}( f(Y_1) - f(Y_0)  \geq z \,\big|\, Y_0=x_0) \leq \mathbb{P} (Z(w_0) \geq z),\qquad \forall z \geq 0.
\end{equation}
Namely, if $f(X^{(\lambda)}_{\tau(x_0)}) - f(X^{(\lambda)}_0) >0$, then $f(Y_1)-f(Y_0)= f(X^{(\lambda)}_{\tau(x_0)}) -f(X^{(\lambda)}_0)$. 
Therefore the bound \eqref{BoundYbyZ} holds by Lemma \ref{hoefd} for $z> 0$. Further, for $z\leq 0$, the right-hand side of Inequality \eqref{BoundYbyZ} is equal to $1$, so the bound trivially holds.

Lemma~\ref{Zmonotone}, stated and proved in the appendix, assists with the coupling of $W$ and $Y$ by providing a monotonicity property for the transitions of $W$. Specifically, for constants $v$, $w$ with $w^*\leq v\leq w$, we have that
\begin{equation}\label{monotone}
\mathbb{P}(W_1  \geq z\,|\,W_0 = v) \leq \mathbb{P}(W_1 \geq z\,|\,W_0 = w)\,.
\end{equation}

Combining together \eqref{BoundYbyZ} and \eqref{monotone}, we have that 
\begin{equation}\label{CouplingBound}
{\mathbb P} (f(Y_1) \geq z\,|\,Y_0 =x_0 ) \leq {\mathbb P} (W_1 \geq z\,|\,W_0=w_0)
\end{equation}
whenever $f(x_0)\leq w_0$. 

A direct consequence of this inequality is that there is a coupling of $f(Y_k)$ and $W_k$ where, provided $f(Y_0) \leq W_0$, then $f(Y_k) \leq W_k$ for all $k$. This short, but standard, argument is presented in the next paragraph.

Let 
\begin{align*}
F_{Y,x_0}(z) &= {\mathbb P}( f(Y_1) \geq z\,|\,Y_0=x_0)\,,\\
F_{Z,w_0}(z) &={\mathbb P} (W_1 \geq z\,|\,W_0=w_0)\,,
\end{align*}
and $U$ be an independent uniform $[0,1]$ random variable. The distribution of $F^{-1}_{Y,x_0}(U)$ and $F_{Z,w_0}^{-1}(U)$ are respectively versions of $f(Y_1)$ and $W_1$ for initial values $Y_0=x_0$ and $W_0=w_0$ (see e.g.\ \cite[Section 3.12]{Wi91}). Thus we set $f(Y_1)=F^{-1}_{Y,x_0}(U)$ and $W_1= F_{Z,w_0}^{-1}(U)$. 

Notice that once $f(Y_1)$ is determined, we can extend the coupling to determine $Y_1$. To do this, we take an independent random variable with distribution
\[
{\mathbb P}( Y_1 = y\,|\,f(Y_1), Y_0)\,.
\]
Now from inequality \eqref{CouplingBound}, it is clear that $F^{-1}_{Y,x_0}(u) \leq F_{Z,w_0}^{-1}(u) $ for all values of $u$. Thus under this coupling
\[
f(Y_1)= F^{-1}_{Y,x_0}(U) \leq F_{Z,w_0}^{-1}(U) = W_1\,.
\]
Continuing inductively, we have $f(Y_k) \leq W_k$ for all $k$, as claimed.
\end{proof}

We now analyze the chain $W_k$. As the following lemma states, we find that its asymptotic drift is zero, whenever the parameter set $\mathcal{L}$ is stable. 

\begin{proposition}\label{WZero} 
\[
\limsup_{k\rightarrow \infty} \frac{W_k}{k} = 0\,.
\]
\end{proposition}
\begin{proof}
It is a straightforward calculation to show that $Z(w)$ is $L^2$ bounded in $w$ and that ${\mathbb E} Z(w) \rightarrow 0$  as $w\rightarrow \infty$. This is shown in Lemma \ref{LemmaZ} in the appendix. We analyze the martingale
\begin{align}
M_k &= \sum_{n =1}^k Z(W_{n-1}) - {\mathbb E}\big[ Z(W_{n-1})\,|\,W_{n-1}\big] \notag \\
& = W_k -W_0 - \sum_{n =1}^k {\mathbb E}\big[ Z(W_{n-1})\,|\,W_{n-1}\big]\,. \label{MLabel}
\end{align}
Since $Z(w)$ is $L^2$ bounded, $M_k$ is an $L^2$ martingale (with unbounded variation). Further, such $L^2$ martingales obey the strong law of large numbers, that is
\begin{equation}
\lim_{k\rightarrow \infty} \frac{M_k}{k} = 0\,.\label{Mg}
\end{equation}
For instance, see  \cite[Section 12.14]{Wi91} for a proof. 

We therefore have 
\begin{align}
\limsup_{k\to\infty} \frac{W_k}{k} &= \limsup_{k\to\infty}\left( \frac{W_0+M_k}{k}+\frac{1}{k}\sum_{n=1}^k {\mathbb E}\big[ Z(W_{n-1}) \,|\,W_{n-1}\big]\right)\notag\\
& \le \limsup_{k \to \infty} \frac{W_0+M_k}{k} + \limsup_{k\to\infty}\frac{1}{k} \sum_{n=1}^k {\mathbb E}\big[ Z(W_{n-1}) \,|\,W_{n-1}\big]\notag\\
&= \limsup_{k\to\infty} \frac{1}{k} \sum_{n=1}^k {\mathbb E}\big[ Z(W_{n-1}) \,|\,W_{n-1}\big]\,, \label{limsupZ}
\end{align}
where the first equality holds due to \eqref{MLabel} and the final equality holds due to \eqref{Mg}. 

We now note that the inequality \eqref{limsupZ} can only hold when $\lim_{k\rightarrow\infty} W_k/k =0$. To see this, note that if $\limsup_k W_k/k$ were positive then $W_k$ must diverge. However, as was shown in Lemma \ref{LemmaZ}, we also have that ${\mathbb E}\big[ Z(W_{n-1}) \,\big|\, W_{n-1}\big]\rightarrow 0$ as $W_{n-1} \rightarrow \infty$. Thus the average of these terms must be zero, that is
\[
0=\limsup_{k\rightarrow \infty} \frac{1}{k} \sum_{n=1}^k {\mathbb E}\big[ Z(W_{n-1})\,\big|\, W_{n-1}\big] \geq \limsup_{k\rightarrow \infty} \frac{W_k}{k} > 0\,,
\]
which is a contradiction. Thus, $\limsup_{k\rightarrow\infty} W_k/k =0$, as required.
\end{proof}
The proof of Theorem~\ref{thm:stable} is now an application of Proposition~\ref{Coupling} and Proposition \ref{WZero}. \\

\begin{proof}[Proof of Theorem \ref{thm:stable}]
For $W_0 = f(Y_0)$ Proposition~\ref{Coupling} provides
\[
f(Y_k) \leq W_k\,,\qquad\text{for all }k\,.
\] 
The time increment $\tau(x)$ is bounded below, so $T_k\geq \gamma k$ for some positive constant $\gamma$. Hence, Proposition \ref{WZero} implies
\[
\limsup_{k\rightarrow \infty} \frac{f(Y_k)}{T_k} \leq  \limsup_{k\rightarrow \infty} \frac{W_k}{\gamma k} = 0\,,
\]
as required.
\end{proof}

\begin{remark}[Adapting the proof to local-search]
We briefly remark one way in which the above argument can be adapted to the local-search case. Firstly, we note that local search (in the worse case) will choose the maximum of two independent simulation runs. Given that both parameters $\gamma$ and $\lambda$ are stable and given Lemma \ref{hoefd}, we then have that the local search update can be bounded as follows
\begin{align*}
 \mathbb P (f(Y_1) \geq z\,|\,f(Y_0) =x, \lambda, \gamma) 
&\leq {\mathbb P} (\left\{ f(X^{(\lambda)}_\tau) - f(x)  \geq z \right\} \cup \left\{ f(X^{(\lambda)}_\tau) - f(x) \geq z \right\} ) \\
&\leq  \mathbb P ( Z(f(x)) + Z'(f(x)) \geq 2z)\,.
\end{align*}
In the second inequality above, we apply Lemma \ref{hoefd} to obtain two i.i.d. copies of $Z(f(x))$. From this one can see that the result of the proof of Theorem~\ref{thm:stable} follows by replacing \eqref{defW} with 
\[
W_k = W_{k-1} + Z(W_{k-1}) + Z'(W_{k-1})
\]
for two iid versions of $Z$. This gives one straightforward way of adapting the proof of Theorem~\ref{thm:stable}. Other methods with tighter bounds are also possible.
\end{remark}

\subsection{Proof of Theorem~\ref{thm:unstable} for the global search algorithm}\label{UNS}
In this subsection we prove Theorem \ref{thm:unstable} for the global search algorithm. The proof relies on two lemmas, Lemma~\ref{SubMg}, and Lemma~\ref{Concentrate}. In Lemma~\ref{SubMg} we bound the drift of $f(Y)$ and show that $f(Y)$ can be used to construct a submartingale.  This follows from the fact that unstable parameters choices significantly increase the drift, while stable parameter choices do not significantly decrease it. In Lemma~\ref{Concentrate} we bound the moments of this submartingale. Then, using standard martingale arguments we show that every time the submartingale exceeds some level, with positive probability it stays above this level forever. Since $f(Y)$ is an irreducible Markov chain our divergence result then follows.  

In the following we use the notation $[x]_+ := \max\{x,\,0\}$ and $\Delta f(x,\,y) := f(y) - f(x)$. 
\begin{lemma}\label{SubMg}
If $\mathcal L$ is unstable, then there exist constants $\kappa \ge 0$ and $a>0$ such that for all $x$ with $|x|\geq \kappa$ 
\begin{equation}
{\mathbb E}\left[ f(Y_{k+1})-f(Y_k)\,|\,Y_k\right] 
\geq 
a\,\tau(Y_k) > 0 \,. \label{Drift}
\end{equation}
\end{lemma}
\begin{proof}
Aside from the simulation in $\mathcal X$ between $Y$ samples, our algorithm consists of two random steps:~ (i) the selection of $\Lambda_{k+1}$, and (ii) the random state update rule \eqref{ChangeRule2}. Upon conditioning on these two steps the expected change in $f$ can be calculated as follows
\begin{align}
&{\mathbb E}\left[\Delta (Y_{k},\,Y_{k+1})\,|\,Y_k=x\right] \notag\\
&= \frac{1}{|{\mathcal L}|}\int_{\mathcal L} {\mathbb E}\left[\Delta f(Y_{k},\,Y_{k+1})\,|\,Y_k=x, \Lambda_{k+1}=\mu \right] d\mu \notag \\
&= \frac{1}{|{\mathcal L}|}\int_{\mathcal L} 
	{\mathbb E}\left[ 
    	 \Delta f(X^{(\mu)}_0,\,X^{(\mu)}_{\tau(x)})\,
			\exp\left(
            		-\eta \left[
                    						-\Delta f(X^{(\mu)}_0,\,X^{(\mu)}_{\tau(x)})
                                            \right]_+ 
                   \right)
\right]{\rm d}\mu\,. \label{ChangeInF}
\end{align}

Denote $p := |\bar{\mathcal L}|/|{\mathcal L}|\in(0,1]$, where  $\bar{\mathcal{L}}$ is the set for which $X_\lambda$ is unstable (cf.\  \eqref{unstable}). 
Now split the above integral by distinguishing between: (i)~$\mu\in \bar{\mathcal L}$, and (ii)~$\mu\in{\mathcal L}\setminus \bar{\mathcal L}$.
It is readily verified that for all $z\in{\mathbb R}$ the function $z\mapsto z\, \exp(-\eta\,[-z]_+)$ satisfies
\begin{equation}\label{ineq}
z\,\exp(-\eta[-z]_+) \ge \max\left\{z,-(\e\,\eta)^{-1}\right\}\,.
\end{equation}
The stated bound is trivial for $z\ge 0$, and for $z<0$ simply note that $z\,\exp(-\eta[-z]_+)$ is minimized at $z=
-\eta^{-1}$. 

For $\mu\in \bar{\mathcal{L}}$ we use the lower bound of $z$ in \eqref{ineq},
\begin{align}
&\frac{1}{|{\mathcal L}|}\int_{\bar{\mathcal L}} {\mathbb E}\left[\Delta f\left(X^{(\mu)}_0,\,X^{(\mu)}_{\tau(x)}\right)
\exp\left(-\eta \left[-\Delta f\left(X^{(\mu)}_0,\,X^{(\mu)}_{\tau(x)}\right)\right]_+ \right)
\right]{\rm d}\mu \notag \\
\ge &
\frac{1}{|{\mathcal L}|}\int_{\bar{\mathcal L}} {\mathbb E}\left[ \Delta f\left(X^{(\mu)}_0,\,X^{(\mu)}_{\tau(x)}\right)\right]
{\rm d}\mu 
 \ge p\,\delta\,\tau(x)\,. \label{LowerBound1}
\end{align}
In the second inequality above, we apply the assumption that our Markov chain is unstable on $\mathcal L$ (cf.\ \eqref{unstable}).

For $\mu\in\mathcal{L}\setminus\bar{\mathcal{L}}$,  we use the lower bound of $-(\e\,\mu)^{-1}$ in \eqref{ineq}. This yields
\begin{align}
&\frac{1}{|{\mathcal L}|}\int_{{\mathcal L}\setminus \bar{\mathcal L}} {\mathbb E}\left[\Delta f\left(X^{(\mu)}_0,\,X^{(\mu)}_{\tau(x)}\right)
\exp\left(-\eta \left[-\Delta f\left(X^{(\mu)}_0,\,X^{(\mu)}_{\tau(x)}\right)\right]_+ \right)
\right]{\rm d}\mu \notag\\
&\ge -\left(1-p\right)(\e\,\mu)^{-1}\,.\label{LowerBound2}
\end{align}

Combining Equation \eqref{ChangeInF} with Inequalities \eqref{LowerBound1} and \eqref{LowerBound2} yields
\begin{equation}\label{ExpectedDrift}
{\mathbb E}\left[\Delta f(Y_k,\,Y_{k+1})\,|\,Y_k=x\right] 
\geq 
p\,\delta\,\tau(x)-\left(1 - p \right)(\e\,\mu)^{-1}\,.
\end{equation}
Since $\tau(x)\rightarrow\infty$ as $|x|\rightarrow \infty$. There exists $\kappa>0$ such that $p\,\delta\,\tau(x)-\left(1 - p \right)(\e\,\mu)^{-1} \geq p\,\delta \tau(x)/2$ for all $|x|>\kappa$.  Letting $a=p\,\delta/2$, we have the result. 
\end{proof}

An immediate consequence of the above proof is that the process
\[
F_k := f(Y_k)-a\sum_{i=0}^{k-1} \tau(Y_i)
\]
forms a submartingale. This in itself is not sufficient to prove $f(Y_k)$ diverges in the sense of Theorem \ref{thm:unstable}. However, this is possible when we bound the moments of $F_k$ as follows. In the following let $S_k = (Y_k,\,\Lambda_k)$.

\begin{lemma}\label{Concentrate}
If $\mathcal L$ is unstable, then there exist an $r>0$ such that for $|Y_{k-1}|\geq \kappa$
\begin{equation}
{\mathbb E}\big[ \exp\big((-r\,(F_{k}-F_{k-1})\big)\,\big|\,S_{k-1}\big] < 1\,. \label{Einf}
\end{equation}
\end{lemma}
\begin{proof}
The change in the process from $F_{k-1}$ to $F_{k}$ is achieved by a process with bounded increments. Upon applying the Azuma--Hoeffding inequality, we have that 
\begin{equation}\label{FAzzuma}
{\mathbb P}  \left(  F_{k} - F_{k-1} -  {\mathbb E} [ F_{k} - F_{k-1}]  \leq -y  \; \Big|  S_{k-1}\right) \leq  \exp\left(- \frac{2\,y^2}{\tau_k\phi_f^2}\right)\,.
\end{equation} 

Now consider the following sequence of inequalities:
\begin{align*}
&\quad {\mathbb E} \left[ \exp\big( -r\,(F_{k} - F_{k-1} ) \big)\,\Big|\,  S_{k-1}\right] \\
&= 
\int_0^\infty 
	{\mathbb P} 
    		\left( 
            		\exp(-r\,(F_{k}-F_{k-1})) \geq z 
                   \Big|
                   S_{k-1} 
                   \right)
               \d z 
\\
&=
\int_0^\infty 
		{\mathbb P}
        		\left( 
                		F_{k}-F_{k-1} \leq -  \frac{1}{r} \log z 
                        \Big|
                       S_{k-1} 
                \right) \d z 
\\
&\leq 
\int_0^\infty 
		{\mathbb P} 
        		\left( 
                		F_{k}-F_{k-1} -{\mathbb E} [ F_k - F_{k-1}]  \leq -a\tau_k -  \frac{1}{r}\log z 
                        \Big|
                        S_{k-1}
                 \right)
                 \d z 
\\
&\leq
\int_{\exp(-ra\tau_k)}^\infty 
		{\mathbb P}
        		\left( 
                		 F_{k}-F_{k-1} -{\mathbb E} [ F_k - F_{k-1}] \leq -a\tau_k -   \frac{1}{r}\log z  
                 		\Big|
                        S_{k-1}
                 \right) 
                 \d z  
          + 
          \e^{-ra\tau_k }
\\
&\leq 
\int_{\exp(-r\,a\,\tau_k)}^\infty \exp\Big(- \frac{2}{\tau_k\,\phi_f^2}(a\,\tau_k +r^{-1}\log z)^2\Big)\d z + \exp(-r\,a\,\tau_k)\,.
\end{align*}
In the first inequality above, we apply the bound that $\mathbb{E} [ F_k - F_{k-1} ] \geq a\,\tau_k$ from Lemma \ref{SubMg}. In the second inequality, for values of $z$ such that $-a\,\tau_k - r^{-1}\log z \geq 0$ we bound the integrand from above by $1$, which results in the $\exp(-r\,a\,\tau_k)$ term appearing. In the final inequality we apply \eqref{FAzzuma}.

We now show that the right hand side of the expression above is strictly less than 1 for a suitable choice of $r$, and $\tau_k$ suitably large:
\begin{align*}
&\int_{\exp(-r\,a \tau_k)}^\infty 
	\e^{
    	- \frac{2}{\tau_k\,\phi_f^2}(a\,\tau_k + r^{-1} \log z)^2
        } 
     \d z 
=  \frac{1}{r} \int_0^\infty 
			\e^{-\frac{2\,y^2}{\tau_k\,\phi_f^2}} 
  		 	\cdot \e^{r\,y} 
            \cdot \e^{-r\,a\, \tau_k}
        \d y \\
&= \frac{1}{r} \int_0^\infty 
			\exp\left(-\frac{2}{\tau_k\,\phi_f^2} (y-r\,\tau_k\phi_f^2/4)^2\right)
           \cdot  \exp\left(r^2\,\tau_k/4\right)\cdot\exp\left( - r\,a\, \tau_k\right)
          \d y \\
&\leq  r^{-1}\sqrt{\tau_k\,\phi_f^2\,\pi/2}\,\cdot \exp(r^2 \tau_k/4 - r\,a\, \tau_k)\,.
\end{align*}
The final inequality follows by integrating over $\mathbb{R}$ rather than $\mathbb{R}_+$ and by noting that the integral of $\exp(-y^2)$ over $\mathbb{R}$ is equal to $\sqrt{\pi}$.

Observe that there exists $r>0$ such that $r^2/4 - r\,a<0$. Thus for this choice of $r$, for all $\tau_k$ suitably large, we have that, as desired,
 \[
 \int_{\exp(-r\,a\, \tau_k)}^\infty 
	\exp\Big(
    	- \frac{1}{\tau_k}(a\,\tau_k + r^{-1} \log z)^2
        \Big)
     \d z
+
\exp(-r\,a\, \tau_k) < 1\,.
\]
\end{proof}

We can now prove Theorem \ref{thm:unstable} using well known martingale arguments.\\

\begin{proof}[Proof of Theorem \ref{thm:unstable}] We first apply standard stopping arguments to \eqref{Einf} to show that if $Y_0$ is such that $|Y_0|>\kappa$, then there is positive probability that $F_k$ will not go negative, namely,
\begin{equation}\label{PFk}
 {\mathbb P} \Big( \inf_{k\ge 0} F_k \geq 0 \Big) \geq 1 - \exp(-rK) >0\, ,
\end{equation}
for some $K>0$. We do so by investigating the probability of its complement. 

Let $T$ be the first time when $F_k < 0$ occurs for $k\ge 0$, which is a stopping time. Using Lemma \ref{Concentrate}, recalling that $r>0$,
\begin{align*}
{\mathbb P} \Big( \inf_{k\ge 0} F_k <  0  \Big) &= {\mathbb P}(F_T < 0)\\
&= {\mathbb P}\left( \e^{-r F_T} > 1 \right) \\
&\le {\mathbb E}\,  \exp(-r F_T) \\
&=  {\mathbb E}\, \Big[  \liminf_{n\rightarrow \infty} \exp(- r F_{T\wedge n}) \Big]\\
&\le \liminf_{n\rightarrow \infty} {\mathbb E}\, \Big[  \exp(- r F_{T\wedge n})\Big] \\
&\le  \liminf_{n\rightarrow \infty} {\mathbb E}\, \exp(- r F_0) = {\mathbb E}\, \exp(- r F_0) \leq \exp(-rK)\, 
\end{align*}
where $K:=\min_{y: |y|>\kappa} \{ f(y)\}$ is a positive constant since $f$ is positive and $f(x)\rightarrow\infty$ as $|x|\rightarrow \infty$. The first two equalities above apply our stopping time definition and an exponential change of variable. The first inequality above applies Markov's inequality, the second applies Fatou's lemma and the third is the optional stopping theorem (see e.g.\ \cite[Section 10.10]{Wi91}) applied to our supermartingale.

The next step is to show using the Strong Markov Property, that at every time $\ell$ when $|Y_\ell|>\kappa$ holds, there is a positive probability that the process $F_k$ remains positive for all remaining time. Due to irreducibility $|Y_k|>\kappa$ occurs infinitely often, and so eventually it will be that $F_{k}>0$ for all time. We now argue this point more formally.
 Let $\ell_0$ be the first time that $|Y_k| > \kappa$ holds. For $n\geq 1$, let 
\[
F^{(n)}_k = f(Y_k) - a\sum_{i=\ell_{n-1}}^{k-1} \tau(Y_i)\,,
\]
which is the process $F$ started from time $\ell_{n-1}$. Let $\sigma_n$ be the first time after $\ell_{n-1}$ when $F^{(n)}_k < 0$ holds, and let $\ell_n$ be the first time after $\sigma_n$ that $|Y_k| > \kappa$ holds.
Since our Markov chain is irreducible it must be that if $\sigma_n$ is finite, then $\ell_{n+1}$ is finite. By this and \eqref{PFk} we have 
\[
\mathbb P(\sigma_n < \infty\,|\,\sigma_{n-1} < \infty)  = \mathbb P(\sigma_n < \infty\,|\,\ell_n < \infty) <  \e^{-rK} .
\] 
Thus, upon noting that $\sigma_n$ cannot possibly be finite if $\sigma_{n-1}$ is not, we have
\[
\mathbb P( \sigma_n < \infty ) \leq \exp(-rK)\,\mathbb P ( \sigma_{n-1} < \infty ) < \dotsc < \exp(-nrK)\,.
\]
Now, note that
\[
\sum_{n=0}^\infty \mathbb P( \sigma_n < \infty ) = \sum_{n=0}^\infty \exp(-nrK) < \infty\,,
\]
so by Borel--Cantelli (see e.g.\ \cite[Section 2.7]{Wi91})
\begin{align*}
\mathbb P ( F^{(n)}_k < 0,\text{ infinitely often }) = 0\,.
\end{align*}
Thus, there exists a $k'$ such that for all $k\ge k'$, we have that 
\[
f(Y_k) - a \sum_{i=k'}^{k-1} \tau(Y_i) \geq 0 
\]
which, after rearranging, implies 
\[
\liminf_{k\rightarrow \infty}\frac{f(Y_k)}{\sum_{i=0}^{k-1} \tau(Y_i)} \geq  \liminf_{k\rightarrow \infty}\frac{f(Y_k)}{\sum_{i=k'}^{k-1} \tau(Y_i)}\cdot\liminf_{k\rightarrow \infty}\frac{\sum_{i=k'}^{k-1} \tau(Y_i)}{\sum_{i=0}^{k-1} \tau(Y_i)}  \geq a\,,
\]
as required.\\
\end{proof}

\subsection{Proof of Theorem~\ref{thm:unstable} for the local search algorithm}\label{Sa}
We now prove Theorem~\ref{thm:unstable} under the premise that the local search algorithm, Algorithm \ref{ALG3}, is applied. First, we consider the situation where the local search algorithm must compare an unstable parameter $\lambda$ with a stable parameter \BP{$\gamma$}. The following lemma will be used to show that the probability of the Metropolis rule, \eqref{LocalChangeRule}, selecting $\gamma$ will be a low probability event.
\begin{lemma}\label{AzLem}
For the events
\begin{align*}
A = 
\Big\{ 
f\big(X^{(\lambda)}_{\tau(x)}\big) - f\big(X^{(\lambda)}_0\big) 
\le 
\frac{3\delta}{4}f(X_0^{(\lambda)})
\Big\}
\quad \text{and} \quad 
B = 
\Big\{ 
f\big(X^{(\gamma)}_{\tau(x)}\big) - f\big(X^{(\gamma)}_0\big) 
\ge 
\frac{\delta}{2} f(X_0^{(\gamma)})
\Big\}\,,
\end{align*}
with $\lambda \in \bar{{\mathcal L}}$ and $\gamma \notin \bar{{\mathcal L}}$ there exists positive constants $\beta_1$ and $\beta_2$ such that
\begin{align}
{\mathbb P} (A \,|\, X_0^{(\lambda)}=x) &\leq \beta_1 \e^{-\beta_2 \tau(x)}\,,\label{ABOUND}\\
{\mathbb P} (B \,|\, X_0^{(\gamma)}=x) &\leq \beta_1 \e^{-\beta_2 \tau(x)}\,.\label{BBOUND}
\end{align}
\end{lemma}
\begin{proof}
The bound \eqref{ABOUND}  is a consequence of the Azuma-Hoeffding Inequality. In particular,
\[
\Big\{ 
f\big(X^{(\lambda)}_{\tau(x)}\big) - f\big(X^{(\lambda)}_0\big) 
\le 
\frac{3\delta}{4}f(X_0^{(\lambda)})
\Big\} 
\subset 
\Big\{ 
f\big(X^{(\lambda)}_{\tau^*}\big) - f\big(X^{(\lambda)}_0\big) 
\le 
\frac{3\delta}{4}f(X_0^{(\lambda)})
\Big\}
\]
where $\tau^*=\min\{ t \leq \tau(x) : |X_t^{(\lambda)}|\leq \kappa\}$ for suitably large values of $|x|$. Since $\lambda$ is unstable, $f(X^{(\lambda)}_{t\wedge \tau^*})$ is a sub-martingale with bounded increments and drift $\delta$. Thus we can directly apply the Azuma-Hoeffding Inequality to obtain \eqref{ABOUND}.

The bound \eqref{BBOUND} is a direct consequence of Lemma \ref{hoefd}. In particular, taking $w=f(x)$ and $z=\frac{\delta}{2} f(x)$, the terms in the exponential in statement \eqref{NastyBound} of Lemma \ref{hoefd} are such that 
\begin{align*}
&  \frac{(z-\alpha_1(w))^2}{2\alpha_2(w)} \sim \left[\frac{(\frac{\delta}{2c}+1)^2}{2(\phi+\delta)^2 \sigma} \right]\tau(x)\,, \\ 
 & \frac{(z -\alpha_3(w))^2}{2 \alpha_4(w)} \sim
 \left[ \frac{(1+\frac{\delta}{2})^2}{2c^2\phi\sigma} \right] \tau(x)\,.
\end{align*}
Further, $n(x)=O(\tau(x))$. This in turn implies that there are constants $\beta_1$ and $\beta_2$ such that \eqref{BBOUND} holds.
\end{proof}

We let $(Y,\Lambda)=(x,\lambda)$ be the initial state of Algorithm \ref{ALG3}, we let $\gamma\notin \bar{\mathcal{L}}$ be the parameter selected in Step (i) of Algorithm \ref{ALG3}, and we let $(Y',\Lambda')$ the state of Algorithm \ref{ALG3} after its first iteration. Given this notation, the following lemma, which is a consequence of the above result, shows that with high probability $\Lambda'=\lambda$ and that over this step $\tau(x)$ is increased by a positive fraction.

\begin{lemma}\label{constpos} There exists positive constants $\epsilon$, $\beta_3$, and $\beta_4$ such that
\[
\mathbb P (\tau(Y') \geq \tau(x)(1+{\epsilon}) , \Lambda'=\lambda) \geq 1- \beta_3 \e^{-\beta_4 \tau(x)}\,.
\] 
\end{lemma}
\begin{proof}
Let $A$ and $B$ be the events specified in Lemma \ref{AzLem}, above. Given the event $A^c$, for $\Lambda'=\lambda$ we have that $f(Y') \geq (1+3\delta/4) f(x)$. Since $f(x) = \Theta( \tau(x))$, for an appropriate choice of $\epsilon>0$ (dependent only on $\delta$), we have that
\[
\tau(Y') \geq  (1+\epsilon )\,\tau(x).
\]
Now given this choice of $\epsilon$ the following equalities hold,
\begin{align*}
&\mathbb P ( \tau(Y') \geq \tau(x) (1+\epsilon ),~\Lambda'=\lambda) \\
\geq & \mathbb P ( \tau(Y') \geq \tau(x) (1+\epsilon ) ,~\Lambda'=\lambda\; |\; A^c, B^c)\,\mathbb P ( A^c \cap B^c)\\
 \geq & \Big(1- \e^{-\frac{1}{4}\delta f(x)}\Big)\Big( 1- 2\beta_1 \e^{-\beta_2 \tau(x)} \Big)
\end{align*}
The second inequality follows from definition of the Metropolis rule, \eqref{LocalChangeRule}, and from Lemma \ref{AzLem}. From this it is clear there are appropriate constants $\beta_3$ and $\beta_4$, as required.
\end{proof}

\begin{proof}[Proof of Theorem~\ref{thm:unstable} for local-search algorithm]
We see that under Assumption \ref{LocalAssump}, the local search algorithm is such that the process $\Lambda_k$ will eventually visit a state in $\bar{\mathcal L}$. 
To see this note that, from any state $(Y_k,\Lambda_k)=(x,\lambda)$ with $\lambda\notin \bar{\Lambda }$, by irreducibility and positive recurrence of $X^{(\lambda)}$ and the fact $\lambda \in \mathcal B_\lambda$, there is a positive probability of reaching state $(x_0,\,\lambda)$. Further, by \eqref{Reduce} there is a positive probability of reaching a state $(x_0,\,\mu)$ for any $\mu \in \bar{\mathcal L}$. From that state, again by the irreducibility of $X^{(\mu)}$, there is a positive probability of reaching a state $x'$ with $\tau(x')> \tau$ for any specified value of $\tau$. Once such a state is reached we now show that there is a positive probability of $\Lambda_k$ remaining in $\bar{\Lambda}$ indefinitely.

Let $E_k$ be the following event 
\[
E_k := \{ \Lambda_k \in \bar{\mathcal{L}},~\tau_k \geq \tau_{k-1}\,(1+\epsilon)  \}\,.
\]
Then, by Lemma \ref{constpos},
\begin{align}
\mathbb P \Big(\bigcap_k E_k \Big) \geq 1 - \sum_k \mathbb P \left( E_k^c\,\Big|\,\bigcap_{k' < k} E_{k'}  \right)  
\geq 1- \sum_k 2 \beta_1 \e^{-\beta_2 \tau_0\,(1+\epsilon)^k}\,.
\end{align}
Thus for suitably large initial values of $\tau_0$ we have that
\[
\mathbb P (\Lambda_k \in \bar{\mathcal L},~\tau_k \geq \tau_{k-1}\,(1+\epsilon)~\forall~k) > 0\,.
\]
Hence, eventually it must occur that the algorithm evolves only according to unstable parameter choices.
\end{proof}

\section{Examples}\label{Cases}
This section presents five example applications of the algorithm, where each example is designed to highlight aspects of the algorithm's implementation and use.  More specifically, we subsequently consider a network of parallel queues, a tandem queueing system, the Rybko--Stolyar network, a network of input queued switches, and a random access network (RAN). 

Throughout the section we use $\mathcal U (A)$ as an indicator variable for the algorithm declaring the set $A$ unstable.

\subsection{Parallel queues with randomly varying connectivity}
For our first example we  extend the illustrative example used in Section~\ref{Sec3}. Consider a system where $N$ parallel queues compete for the service of a single server. Time is slotted, and in each time slot $t \in \mathbb Z_+$ queue $i \in \{1, \dots, N\}$ is connected to the server with probability $0.8$. Similarly, at the beginning of each time slot an arrival occurs at each queue with probability $p \in [0,\,1]$, so that there are at most $N$ arrivals to the system in any particular time slot. After the arrivals have occurred and connectivity is determined, the longest non-zero queue that is connected to the server is reduced by one with probability $\frac{4}{5}$ --- a policy called \emph{longest queue first} (LQF). The system is therefore a discrete time Markov chain $X^{(p)}$ taking values in $\mathbb Z_+^N$. We illustrate this system in Figure~\ref{fig:ParallelSystem}. 

\begin{figure}[h]
	\centering
	\includegraphics{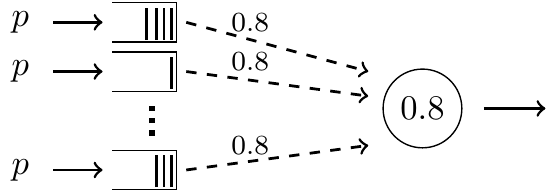}
	\caption{A parallel queueing system with randomly varying connectivity. }
	\label{fig:ParallelSystem}
\end{figure}

The stability region for this irreducible Markov chain is known. In particular, from Corollary~1 in \cite{Tassiulas1993} we have that for any $p \le \ell^*$, where
\[
\ell^* = \frac{4}{5}\left(\frac{1-(1/5)^N}{N}\right)\,,
\]
the limiting distribution of $X^{(p)}$ exists, and otherwise does not. 

Therefore any $\mathcal L \subset [0, \infty)$ that shares an intersection with $[\ell^*, \infty)$ of positive measure, is unstable under our Definition~\ref{StableDef}. Taking $N= 4$ and $\mathcal L$ to be of the form $[0, \ell)$, we therefore have instability for approximately those instances when $\ell > 0.2$, that is $\ell^* \approx 0.1997$. Furthermore, Theorem~1 in \cite{Tassiulas1993} shows that the system is stable under the LQF policy for the network's subcritical region --- a property known as \emph{maximal stability}. This property is well known to hold for single-hop networks under LQF and its generalization the Max Weight-$\alpha$ algorithm (see e.g. \cite{mckeown1999achieving,TaEp91}). 

In Figure~\ref{fig:ParallelK} we give the proportion of simulation runs out of 1000 where the parameter set $[0,\,0.3]$ is declared unstable by the local and global algorithms as $k^*$ is increased. Recall that $k^*$ is the total number of steps the algorithm is permitted to take in $\mathcal X$  before a value of $f(Y_k)$ is compared to $q^{(\alpha)}_k$. Now, the greatest change in $f$ occurs when there are no services and all queues experience an arrival, so that $\phi = 4$.  We assume $\kappa = 4$ and $\sigma = 1$.  It can be seen that longer simulation runs are more likely to declare the system unstable, with an apparent almost sure declaration of instability in the limit. In this case, the local algorithm approaches this limit far more rapidly than the global algorithm. 

\begin{figure}[h]
	\centering
	\includegraphics{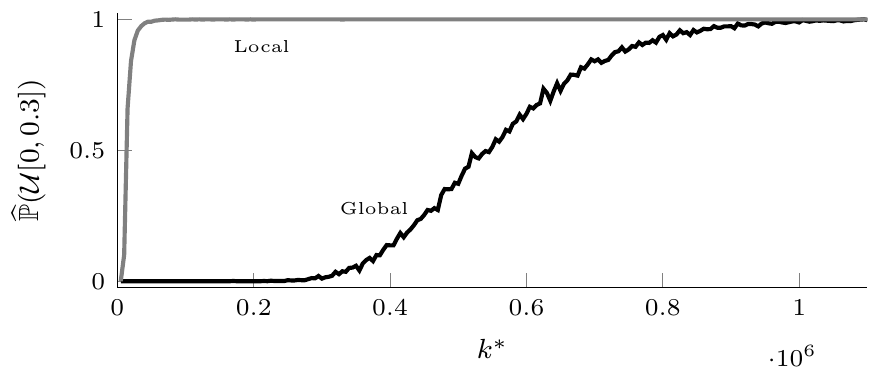}
	\caption{Parallel system $L_1$-stability tests for $p$ sampled from the set $\mathcal L = [0,\,0.3]$ for $k^* \in (0, 10^6]$ with $\tau(x) = 0.5|x|+1$, $\delta = 0.01$, $\sigma = 1$, $\kappa = 4$ and $\epsilon = 0.01$.}
	\label{fig:ParallelK}
\end{figure}

Figure~\ref{fig:ParallelD} explores the effect of the chosen $\delta$ on an unstable declaration. The figure gives the proportion of simulation runs out of 100 where the parameter set $[0,\,0.21]$ is declared unstable by the local and global algorithms. Recall that the definition of stability we use compares the drift of the process under consideration with a linear function that depends on $\delta$. As discussed in Section~\ref{Sec3}, with reference to Figure~\ref{fig:SimpleHypTest}, if a parameter is unstable for a particular $\delta$, then this implies instability for all higher values of $\delta$.  This is because a $W$ process parameterized by a particular $\delta$ will stochastically dominate all $W$ processes parameterized by higher choices of $\delta$. Figure~\ref{fig:ParallelD} demonstrates that this occurs for both the global and local search algorithms. Again we see that the local algorithm appears to perform better ---  in this example it has detected lower values of downward drift when $k^* = 10^5, 10^6$. 

\begin{figure}[h]
	\centering
	\includegraphics{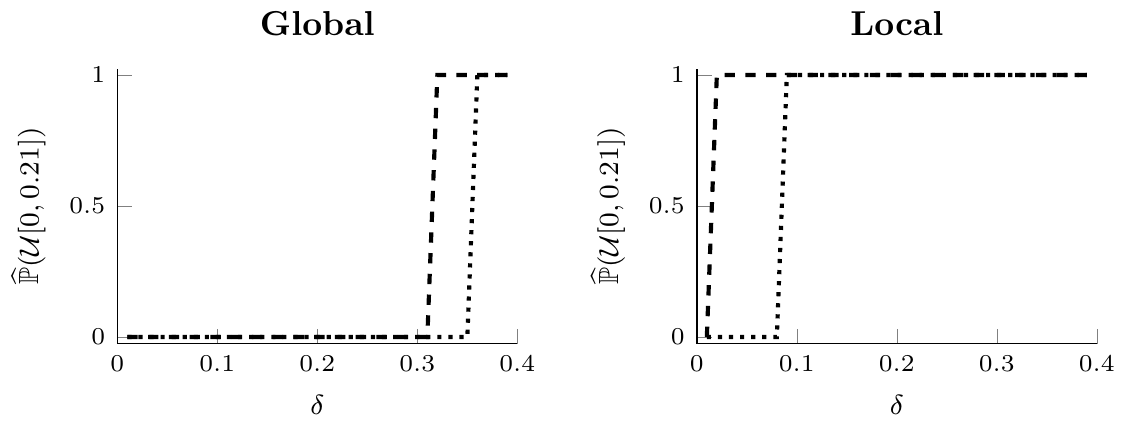}
	\caption{Parallel system $L_1$-stability tests for $p$ sampled from the set $\mathcal L = [0,\,0.21]$ for $\delta \in [0.01,\,0.4]$ with $\tau(x) = 0.5|x|+1$, $\sigma = 1$, $\kappa = 4$, $\epsilon = 0.01$ and $k^* = 10^5$ (dotted), $10^6$ (dashed).}
	\label{fig:ParallelD}
\end{figure}

In Figure~\ref{fig:ParallelL} we give the proportion of simulation runs out of 100 where the  parameter set $[0,\,\ell]$ is declared unstable for a range of $\ell$. It can be seen that longer simulation runs declare the system unstable for a larger proportion of the $\ell$ values that give an unstable $\mathcal L$. The figure provides evidence that in the discrete time case the algorithm is performing as it is intended to, in the next section we move to a continuous time example. 

\begin{figure}[h]
	\centering
	\includegraphics{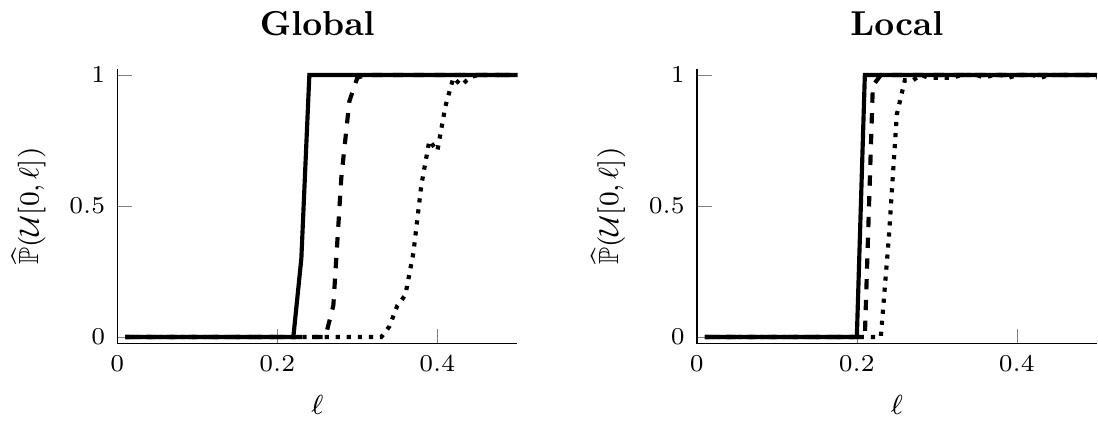}
	\caption{Parallel system $L_1$-stability tests for $p$ sampled from sets of the form $\mathcal L = [0,\,\ell]$ with $\tau(x) = 0.5|x|+1$, $\delta = 0.05$, $\sigma = 1$, $\kappa = 4$, $\epsilon = 0.01$ and $k^* = 10^5$ (dotted), $10^6$ (dashed), $10^7$ (solid).}
	\label{fig:ParallelL}
\end{figure}

\subsection{Tandem queues}
Our next example is the tandem queueing system. We will contrast the results for a Markov system consisting of two $M/M/1$ queues with a system that has renewal arrivals and i.i.d. service times at both nodes (which is not Markov). In the former system jobs arrive to a server according to a Poisson process with rate one, they are then processed one at a time, first come first served (FCFS),  with {\sf Exp}$(\mu_1^{-1})$ service times, before being sent to a subsequent server where they are again processed one at a time, FCFS, with service time {\sf Exp}$(\mu_2^{-1})$. It is well known that the output from the first server to the second corresponds to a Poisson process with rate $\min\{1,\,\mu_1^{-1}\}$. Consequently, the system is $L_1$-stable for $(\mu_1,\,\mu_2) \in [0,\,1]^2$, and $L_1$-unstable otherwise. In the latter system we assume the times between arrivals to the first server are Erlang distributed with rate parameter $1/2$ and shape parameter 2. Jobs are also served FCFS and must pass through the first server before being sent to the second. In this case the service times are Weibull distributed with shape parameter 2, so that they have distribution function $(1-\exp(-(x/\mu)^k)$ for $x\ge 0$, with $k = 2$ and scale parameters $\mu= \mu_1$ and $\mu = \mu_2$ for the first and second server, respectively. Note that in both cases the mean time between arrivals is 1, that the mean service times are $\mu_1$ and $\mu_2$ for the former case, and are $\Gamma(1.5)\mu_1 \approx 0.8862\, \mu_1$ and $\Gamma(1.5)\mu_2 \approx 0.8862\, \mu_2$ in the latter case. 

To apply our discrete time framework to these continuous time systems, we have used the embedded process corresponding to the sequence of states recorded immediately after each jump (which is  Markovian for the $M/M/1$ system, and non Markovian for the system with renewal arrivals and i.i.d.\ service times). In Figure~\ref{fig:TandemMM} and Figure~\ref{fig:TandemGG} we are testing parameter sets of the form $(\mu_1,\,\mu_2) \in \mathcal L = [0,\, \ell]^2$, and as such sets with $\ell > 1$ are $L_1$-unstable in the Markov case and approximately $\ell > 1.1284=(0.8862)^{-1}$ in the non Markov case.  In both the global and local cases it is clear that the test converges to an accurate declaration of instability over $\ell \in (0.5,\, 1.5)$ as $k^* \to \infty$. The figures provides evidence that it is possible to relax the discrete time and Markov assumptions we made in the theoretical development of our algorithm. 

\begin{figure}[h]
	\centering
	\includegraphics{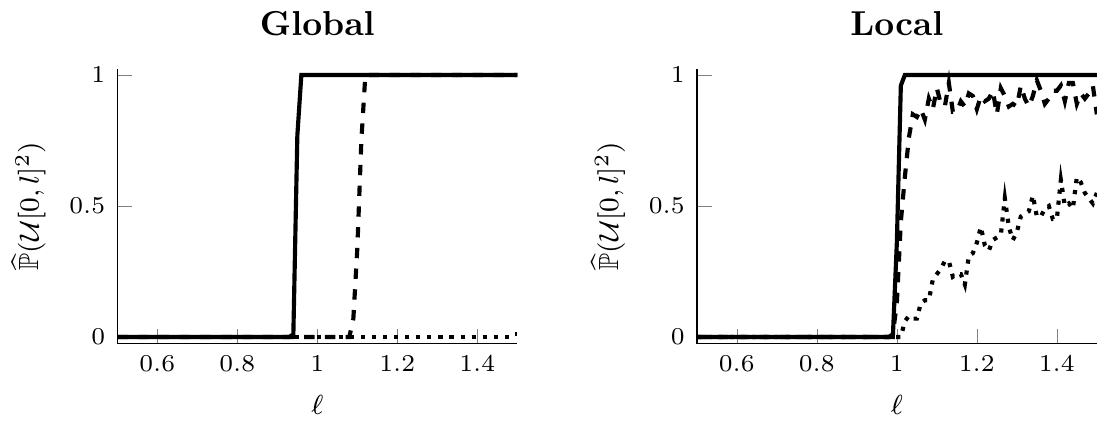}
	\caption{Tandem $M/M/1$ system $L_1$-stability tests for $(\mu_1,\,\mu_2)$ sampled from sets of the form $\mathcal L = [0,\,\ell]^2$ with $\tau(x) = 0.5|x|+1$, $\delta = 0.05$, $\sigma = 1$, $\kappa = 1$, $\epsilon = 0.01$ and $k^* = 10^5$ (dotted), $10^6$ (dashed), $10^7$ (solid).}
	\label{fig:TandemMM}
\end{figure}

\begin{figure}[h]
	\centering
	\includegraphics{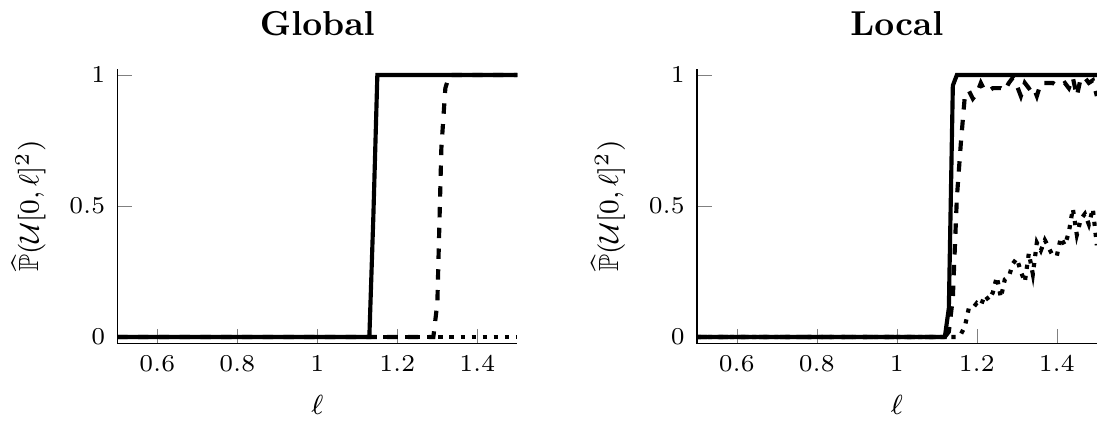}
	\caption{Non-Markovian tandem system $L_1$-stability tests for $(\mu_1,\,\mu_2)$ sampled from sets of the form $\mathcal L = [0,\,\ell]^2$ with $\tau(x) = 0.5|x|+1$, $\delta = 0.05$, $\sigma = 1$, $\kappa = 1$, $\epsilon = 0.01$ and $k^* = 10^5$ (dotted), $10^6$ (dashed), $10^7$ (solid).}
	\label{fig:TandemGG}
\end{figure}

 Further, we are stretching the original modeling framework since there is no fixed $\sigma$ after which the systems exhibit unstable behavior. The required number of steps before an upward drift is expected to occur depends on the system state. Consequently, over short time periods, unstable parameter choices may appear stable, e.g., in the Markov system, when the second server has a very large queue but a parameter selection with $\mu_1>1$  and $\mu_2 < 2-\mu_1$ is made.  Nonetheless, asymptotically both systems are expected to become infinitely large due to the first queue being unstable, and through the $\tau$ function our algorithm is able to maintain accurate prediction.  Due to this, in systems of this kind the choice of $c$ in the $\tau$ function may have an important impact on the algorithm's performance. 

In Figure~\ref{fig:TandemC} we perform instability tests on $[0,\,1.2]$ for a range of $c$. For the global algorithm the choice of $c$ can have a substantial impact on performance, for $k^*=10^6$ a high value of $c$ is required to obtain a high level of accuracy. For the local algorithm, however, the choice of $c$ does not appear to have as much of an effect as the choice of $k^*$. This suggests that if $k^*$ is limited by computational resources, then it is preferable to use the global algorithm with a high $c$ --- particularly if the system is suspected of exhibiting oscillatory behavior. 

\begin{figure}[h]
	\centering
	\includegraphics{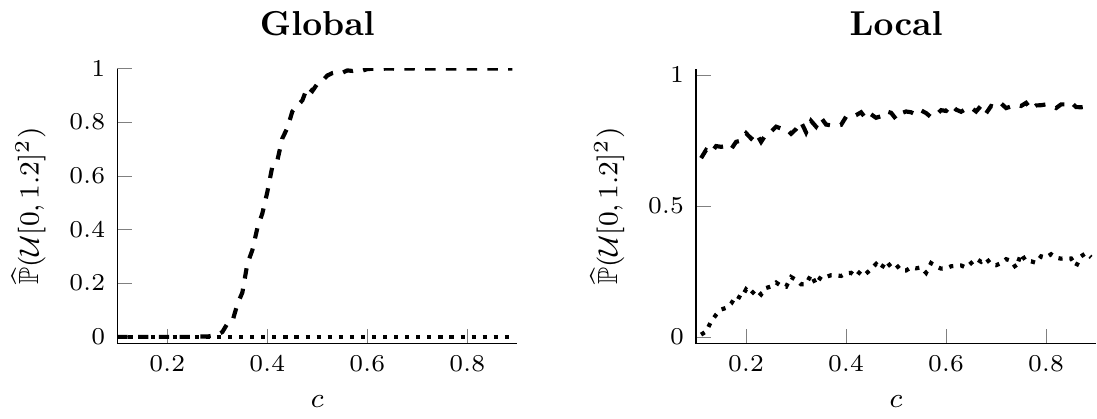}
	\caption{Tandem $M/M/1$ system $L_1$-stability tests for $(\mu_1,\,\mu_2)$ sampled from $\mathcal L = [0,\,1.2]^2$ with $\tau(x) = 0.5|x|+1$, $\delta = 0.05$, $\sigma = 1$, $\kappa = 1$, $\epsilon = 0.01$, $k^* = 10^5$ (dotted), $10^6$ (dashed), and $c \in (0.1, 0.6)$.}
	\label{fig:TandemC}
\end{figure}

\subsection{Rybko--Stolyar queueing network}
The Rybko--Stolyar queueing network, displayed in Figure~\ref{fig:RS}, was introduced in \cite{RySt93} as an example of a work-conserving queueing network that can be unstable for sub-critical parameter choices.  To the best of our knowledge, matching necessary and sufficient conditions for instability are not known.

This queueing network consists of two stations, each with a single server, which we call the left and right stations. All customers served at the left station require ${\sf Exp}(\mu_l)$ service time and all customers served at the right station require ${\sf Exp}(\mu_r)$ service time. There are two classes of customers. The first class enters the network according to a Poisson process at rate $\lambda$ where it is served  at the left station before proceeding to the right station to be served, and  from here it departs the network. Jobs from the second class also enter the network at rate $\lambda$, are served at the right station, proceed to be served  at the left station, and then depart from the network.  Within each customer class the customers are served on a FCFS basis. Between the customer classes, however, there is priority:  jobs being served at their second station (bold in Figure~\ref{fig:RS}) have priority over jobs being served at their first station. 

In \cite{RySt93} it is shown that for $\lambda$ equal to one and $\mu_r>0$, a sufficient condition for instability is $\mu_l < 2$.  In Figure~\ref{fig:RSHypTest} we consider the situation where $\mu_l$ is sampled from sets of the form $(\ell, \ell+1)$ for $\ell \in (1,3)$, with $\lambda = 1$ and $\mu_r = 4$. Due to the result from  \cite{RySt93} we expect that $\ell \in (1,2)$ will be returned as unstable by the algorithm. This occurs for $k^*$ equal to $10^7$. Interestingly, for $\ell > 2$ we never reject the null hypothesis of stability, suggesting that $\mu_l < 2$ is also a necessary condition for instability with $\lambda = 1$ and $\mu_r>0$. In this case the local algorithm appears to outperform the global algorithm. The estimates for the local algorithm do, however, exhibit a large amount of variance (over the 100 sample paths used to generate the figure). 

\begin{figure}[h]
	\centering
	\includegraphics{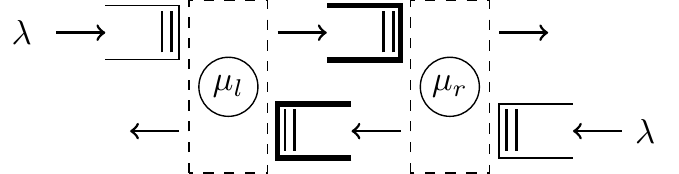}
	\caption{The Rybko--Stolyar network. }
	\label{fig:RS}
\end{figure}

\begin{figure}[h!]
	\centering
	\includegraphics{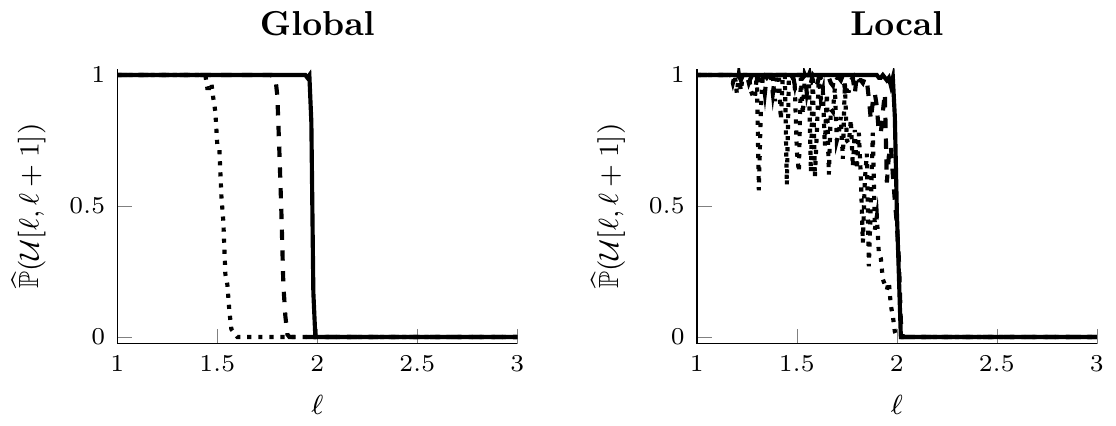}
	\caption{Rybko--Stolyar system $L_1$-stability tests for $\mu_l$ sampled from sets of the form $\mathcal L = [\ell,\,\ell +1]$ for $k^*= 10^5$ (dotted), $10^6$ (dashed), $10^7$ (solid) with $\lambda = 1$, $\mu_r = 4$, $\tau(x) = 0.5|x|+1$, $\delta = 0.05$, $\phi = \kappa = \sigma =1$ and $\epsilon = 0.01$.}
	\label{fig:RSHypTest}
\end{figure}
\subsection{A switch network}
Our next example is a network of input-queued switches which was investigated by Andrews and Zhang in \cite{AndrewsZhang2003}. This discrete time model provides an example where the LQF policy is not maximally stable. In this simulation study, we are able to demonstrate the use of our algorithm on a model which exhibits complex queueing dynamics on a 52 dimensional state space. Again, unlike the parallel queue or tandem models considered earlier, the explicit form of the stability region of this model is unknown. 

The model we are considering is illustrated in Figure~\ref{fig:Switch}.  It has four main switches with labels $A$, $B$, $C$, and $D$ and four auxiliary switches with labels $A'$, $B'$, $C'$, and $D'$. Each of the main switches has ten external input queues to which a packet arrival occurs instantaneously at the beginning of each time slot independently and with probability $r/30$. 

Packets are given a type according to the switch at which they first arrive, for example packets starting at $A$ are of type 1; packets are routed through the network according to their type.  After these arrivals the longest of the 12 queues at each main switch and of the three queues at each auxiliary switch sends a single packet to the corresponding input queue of another switch or are removed from the system (as designated by Figure~\ref{fig:Switch}). Packets sent in a time slot arrive at their destination at the beginning of the next time slot. 

\begin{figure}[h]
	\centering
	\includegraphics{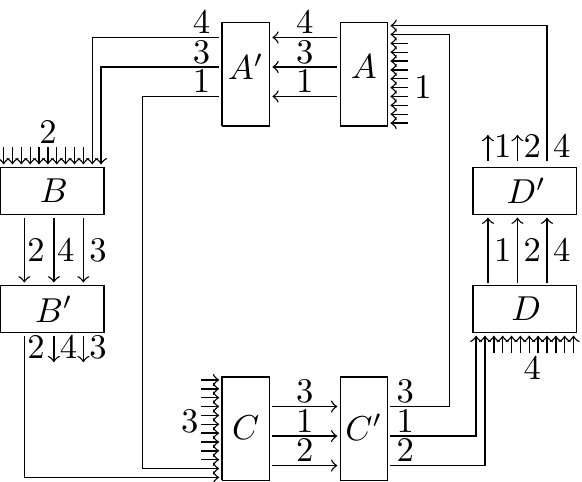}
	\caption{A network of input queued switches. }
	\label{fig:Switch}
\end{figure}

In Figure~\ref{fig:SwitchHypTest} we test for $L_1$-instability in $r$ on parameter sets of the form $\mathcal L_s = [0.5, \ell]$. Due to the large size of the system we have chosen $\delta = 5$. We set $\phi = 40$,   $\tau(x) = 0.5\,|x|+1$, and  $\kappa = \sigma = 1$. Although the stability region for this model is not yet known, this figure provides strong (statistical)  evidence that the set $[0,\,0.95]$ is unstable. We have thus demonstrated that our algorithm can be used to provide statistical evidence that the LQF policy is not necessarily maximally stable in multi-hop settings. In this case the global algorithm appears to perform much better, suggesting that $k^*=10^7$ is not great enough for the local algorithm to start performing well. 

\begin{figure}[h]
	\centering
	\includegraphics{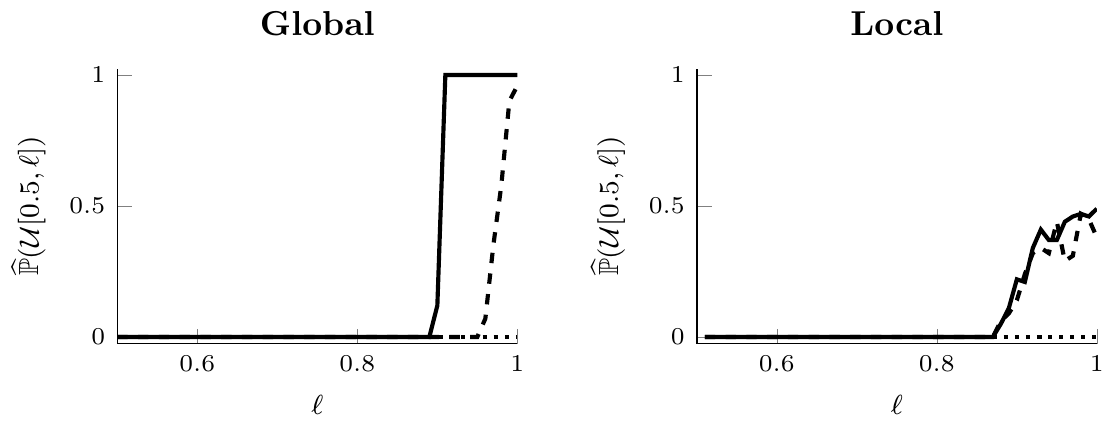}
	\caption{Network of input queued switches $L_1$-stability tests for $r$ sampled from sets of the form $\mathcal L = [0,\,\ell]$ or $\mathcal L = [0.5,\,\ell]$ for $k^*= 10^5, 10^6, 10^7$.}
	\label{fig:SwitchHypTest}
\end{figure}

It may be the case the ratio $|\bar{\mathcal L}|/{\mathcal L}$ has a substantial impact on performance in finite time. Figure~\ref{fig:SwitchHypTest2} explores this relationship by testing for stability of $[\ell, 0.95]$ over a variety of $(k^*, \ell)$ combinations. Intuitively, this ratio should have a greater impact on performance of the local algorithm than the global algorithm. Instead, the figure indicates highly similar (poor) performance over the varying combinations of $(k^*, \ell)$, with some degradation of accuracy for very low $k^*$. While for the global algorithm fixing either $k^*$ or $\ell$ and then increasing the other leads to substantial increases in accuracy.  

\begin{figure}[h]
	\centering
	\includegraphics{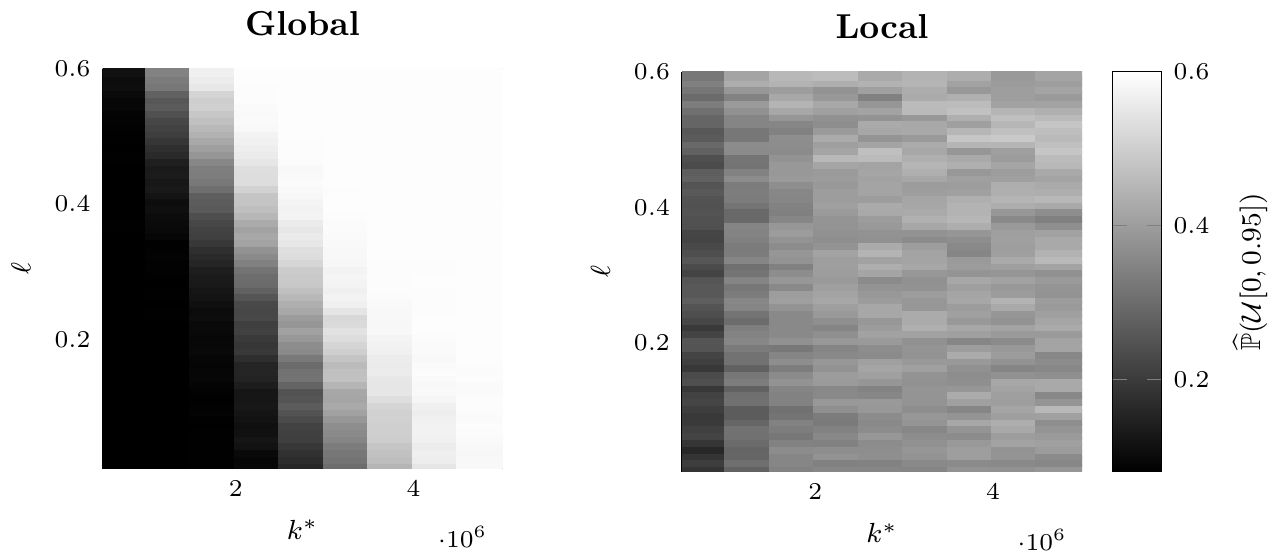}
	\caption{Network of input queued switches $L_1$-stability tests for $r$ sampled from sets of the form $\mathcal L = [\ell,\,0.95]$ for $k^*\in(0,\,7\cdot 10^6]$,  $\tau(x) = 0.5|x|+1$, $\delta = 5$, $\phi = 40$, $\kappa = \sigma = 1$.}
	\label{fig:SwitchHypTest2}
\end{figure}

\subsection{A broken diamond random access network}
So far we have presented classical examples that facilitated the assessment of the algorithm's performance. In our final example we   address a contemporary area of research initiated by \cite{Ghaderi2014}, exploring the stability properties of a wireless network with a queue-based random-access algorithm.  
We focus on a network consisting of nodes $\{1, 2, \dots, 6\}$, some of which are connected by edges, as depicted in Figure~\ref{fig:RAN} (where it is remarked that \cite{Ghaderi2014}
is set in a more general context). In our model we assume that nodes which are connected by an edge interfere with each other, that is, they cannot transmit simultaneously. 

\begin{figure}[h]
	\centering
	\includegraphics{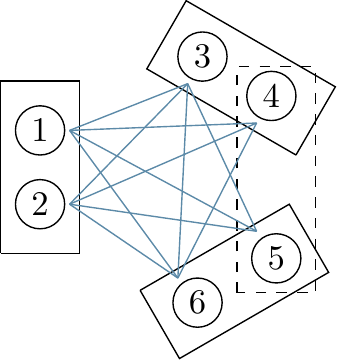}
	\caption{A broken diamond random access network. }
	\label{fig:RAN}
\end{figure}

In this continuous time model, packets arrive to node $i$ according to a Poisson process with rate $\lambda_i$ and take ${\sf Exp}(\mu_i)$ time to transmit, so that the traffic intensity at node $i$ is $\rho_i = \lambda_i/\mu_i$. Let $U(t) \in \{0,1\}^6$ be a vector of indicator variables representing which nodes are active at time $t$ and $X(t) \in \mathbb \{0, 1, \dots\}^6$ be a vector representing the number of packets at each node at time $t$. 

In order to fully describe the evolution of this process, we must specify how nodes decide when to attempt transmission of packets. Whenever a node is not being interfered with it will wait an ${\sf Exp}(\nu_i)$ amount of \emph{back-off period}.  At this point, it will then begin transmitting with probability $\phi_i(X_i(t))$, where $\phi_i(0) = 0$, and otherwise it will begin another back-off period with the same distribution. After each successful transmission, node $i$ will release the medium and begin a back-off period with probability $\psi_i(X_i(t^-))$, with $\psi_i(1)=1$ for all $i$, and otherwise begin another transmission. 

It is easy to see that $(X, U)$ is a Markov process evolving according to the rates given in Table~\ref{tab:RAN}. Note that here $\bar u_i = 0$ indicates that none of the neighbors of $i$ is transmitting.

\begin{table}[h]
\caption{Transition rates of the random access network network in Figure~\ref{fig:RAN}. }
\begin{center}
\begin{tabular}{l |c | c}\label{tab:RAN}
\hspace{7mm}Transition & Rate & States\\
\hline \hline
$(x,\, u) \rightarrow (x+e_i,\, u)$ & $\lambda_i$ & All\\
$(x,\,u) \rightarrow (x,\,u+e_i)$ & $\nu_i\,\phi_i(x_i)$ & $x_i >0$, $u_i=0$, $\bar u_i = 0$\\
$(x,\, u) \rightarrow (x-e_i,\,u)$ & $\mu_i\,(1-\psi_i(x_i))$ & $x_i \ge 1$, $u_i=1$\\
$(x,\,u) \rightarrow (x-e_i,\, u-e_i)$ & $\mu_i\,\psi_i(x_i)$ & $x_i \ge 1$, $u_i=1$\\
\hline \hline
\end{tabular}
\end{center}
\end{table}

Consider the network in Figure~\ref{fig:RAN}, and suppose that $\phi_i(x) \equiv 1$, $x>0$, and $\psi_i(x) = \mathrm{o}(x^{-\gamma})$, with $\gamma > 1$.  Let 
\[(\rho_1, \rho_2, \rho_3, \rho_4,\rho_5, \rho_6) = \rho\,(\kappa_1, \kappa_2, \kappa_3, \kappa_3-\alpha, \kappa_6-\alpha, \kappa_6),\] with $(\kappa_1 \vee \kappa_2)+\kappa_3+\kappa_6 = 1$, and $0< \alpha < (\kappa_3 \wedge \kappa_6)$. Then the main result of Ghaderi et al.\ in \cite{Ghaderi2014} implies that there exists a constant $\rho^*(\kappa, \alpha) < 1$, such that for all $\rho \in (\rho^*(\kappa, \alpha),\,1]$ the Markov process is transient under the given parameter conditions. 

We now consider the example network from the simulation section of \cite{Ghaderi2014}. The relative traffic intensities are taken to satisfy $\kappa_1 = \kappa_2 = \kappa_3 = 0.4$ and $\kappa_6 = 0.2$ with $\alpha=0$. Further, $\phi_i(x) \equiv 1$, $x\ge1$, and $\psi_i(x) = (1+x)^{-2}$. The authors note that it is `difficult to make any conclusive statements concerning stability/instability based on simulation results alone'. They do, however, remark that for these parameter choices and $\rho = 0.97$, their simulated sample paths appear to demonstrate strong signs of instability. 

\begin{figure}[h]
	\centering
	\includegraphics{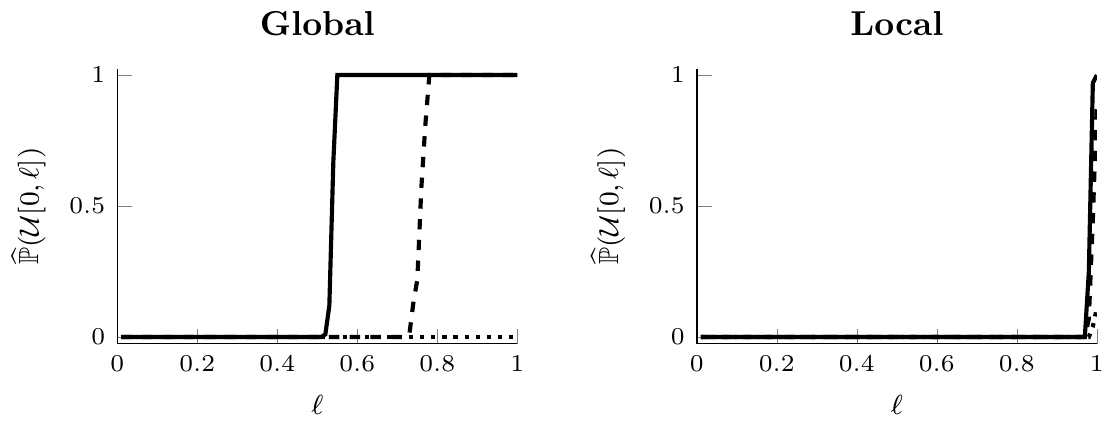}
	\caption{Broken diamond random access network $L_1$-stability tests for sets of the form $\mathcal L = [0,\,\ell]$.}
	\label{fig:RANHypTest}
\end{figure}

In order to perform our stability test, we assume $\kappa = \phi = \sigma = 1$. In Figure~\ref{fig:RANHypTest} we test for $L_1$ instability with $\delta = 0.05$. 
Looking at Figure~\ref{fig:RANHypTest}, which uses our simulation based stability test, we are able to say, with a strong statistically firm footing, that there exists a constant $\rho^*(\kappa, \alpha) < 1$, such that for all $\rho \in (\rho^*(\kappa, \alpha),\,\ell]$ the network is unstable for a range of $\ell$ in approximately $[0.6,\,1]$. This statement expands on the statement of the theorem (for a particular choice of parameters) by allowing for more information to be gained about what values are likely to be possible for $\rho^*(\kappa, \alpha)$. Of course, our statement does not rule out perverse behavior such as the network suddenly exploding after $10^7$ jumps of the process.  It can however be very quickly and easily applied to similar or even vastly more complex networks. We note that the global algorithm is in this case enabling us to make this strong statement, while the local algorithm algorithm only allows us to make the statement for a substantially reduced set of $\ell$. 

\section{Concluding remarks}\label{Disc} 
The main contribution of this paper concerned the development of an automated procedure that determines if, for a specified set of parameter values, a given Markov chain is unstable. A distinctive feature of our work is that our method is simulation based, and in addition broadly applicable and straightforward to implement.
It  provides statistical statements on the stability of the parameter set, but, notably, we have succeeded in providing explicit performance guarantees. Some of our experiments show that our technique provides us with useful insights for models for which the stability set has not been characterized so far.  

Our paper can be considered as a pioneering study on this topic, and various extensions and improvements are envisaged. An important first branch of research could relate to relaxing the assumptions imposed, such as the fact that we restrict ourselves to the class of Markov processes and the bounded step size assumption. Experiments that we performed for non Markovian tandem queues indicated that the approach still provides us with the correct result, if we perform our algorithm as if the underlying system is Markovian.  In order to remove the bounded step size assumption it would be necessary to use a concentration inequality that is stronger than Azuma--Hoeffding. Additionally, our experiments contrasted global and local search versions of the algorithm. We obtained mixed results on performance, and were unable to declare either version superior to the other. Determining conditions that point towards which of these versions should be used in different circumstances remains to be a challenge.

The objective of a second branch of research could be to enhance our procedure such that it can identify, in case instability is detected, which components of the multi-dimensional Markov chain are unstable. A third branch is of an empirical nature, and relates to models of which the stability region is not yet known. By performing systematic simulation studies one could possibly state conjectures.

\section*{Acknowledgements}{\small
The authors thank Yoni Nazarathy (The University of Queensland) for suggesting this area of research and for stimulating discussions. BP would like to acknowledge the support of the Australian Research Council (ARC) through the ARC Centre of Excellence for the Mathematical and Statistical Frontiers (ACEMS).  
BP and MM acknowledge support from Gravitation project N{\sc etworks}, grant number 024.002.003, funded by the Netherlands Organisation for Scientific Research (NWO).
NW's research  was partly funded by the VENI research programme, also funded by the NWO.}

\appendix
\section{Lemmas}
\begin{lemma}\label{Zmonotone} There exists a positive constant $w^*$ such that
\[
\mathbb{P}(W_1  \geq z\,|\,W_0 = v) \leq \mathbb{P}(W_1 \geq z\,|\,W_0 = w)\,.
\]
for $v$ and $w$ such that $z \le w^*\le v \le w$.
\end{lemma}
\begin{proof}
Noting the equality
\[
\mathbb P ( W_1 \geq z\, |\, W_0 = w) = \mathbb P(Z(w) \geq z-w),
\]
we claim it is sufficient to prove that the function
\begin{equation}\label{gform}
g(z,w) = \exp\left(- \frac{(z-w+n(w))^2}{bw}\right)
\end{equation}
is nondecreasing in $w$ for values of $z$ with $z\geq w$. This is because the second term in \eqref{NastyBound} evaluated at a point $z-w$, given our assumptions on the function $\tau$, will tend to zero as $z \to \infty$, and the first term is of the form \eqref{gform}. Also note that for $g$ nondecreasing the function $1 \wedge g$ is also nondecreasing. 

After taking logs and rearranging $g(v,\,z)\leq g(w,\,z)$, we see that it is sufficient to show that 
\begin{equation}\label{bounder}
\frac{z+ n(v)-v }{\sqrt{n(v)}} \geq \frac{z+n(w)-w}{\sqrt{n(w)}} 
\end{equation}
for $z \le w^*\le v \le w$. It is finally noted that \eqref{bounder} holds as long as for $w\geq w^*$ we have $n(w) \le w$. 
\end{proof}
\begin{lemma}\label{LemmaZ}
The random variables $Z(w)$ are $L^2$ bounded and\\
\[
{\mathbb E} Z(w) \rightarrow 0\qquad \text{as}\qquad w\rightarrow 0\,.
\]
\end{lemma}
\begin{proof}
In particular, we bound the mean of $Z(w)$ for large values of $w$ as follows:
\begin{align*}
{\mathbb E} Z(w) & = 
\int_0^\infty {\mathbb P}(Z(w)\ge z){\rm d}z \\
&\leq \int_0^\infty \left[ \exp\left(- \frac{(z+a_1n(w) )^2}{b_1 n(w)}\right) + n(w)\,\exp\left(- \frac{(z + a_2 w)^2}{b_2 n(w)} \right) \right] {\rm d} z.
\end{align*}
We start by bounding the first of these terms:
\begin{align*}
& \int_0^\infty  \exp\left(- \frac{(z+a_1 n(w) )^2}{b_1 n(w)}\right) {\rm d} z \\
& = \int_{0}^\infty - \frac{b_1n(w)}{2(z+ a_1n(w))}{\rm d}\left(\exp\left(-\frac{(z+a_1n(w))^2}{b_1n(w)}\right)\right){\rm d} z \\
& = \frac{b_1}{2\,a_1 }\exp\left(-\frac{a_1^2\,n(w)}{b_1}\right) -\int_0^\infty \frac{b_1 n(w)}{2(z+a_1 n(w))^2}\exp\left(- \frac{(z+a_1 n(w) )^2}{b_1 n(w)}\right){\rm d}z\\
& \le \frac{b_1}{a_1}\exp\left(-\frac{a_1^2 n(w)}{b_1 }\right)\,.
\end{align*}
Upon applying a similar sequence of steps to the second term we find that
\[
{\mathbb E} Z(w)  \leq  \frac{b_1 }{a_1 }\exp\left(-\frac{a_1^2 n(w)}{b_1 }\right) +  \frac{b_2 n(w) }{2 a_2 w }\exp\left(-\frac{a^2_2 w^2}{b_2 n(w)}\right)\,.
\]
We conclude that ${\mathbb E} Z(w) \rightarrow 0$  as $w\rightarrow \infty$, as required.

We now analyze the second moment of $Z(w)$ to establish that these random variables are $L^2$ bounded. Observe that
\begin{align*}
\mathbb{E} \Big[ Z(w)^2 \Big] &= \int_0^\infty 2z\,\mathbb{P} \Big( Z(w) \geq z \Big) {\rm d}z \\
&= \int_0^\infty 2z\,\left[ \exp\left(- \frac{(z+a_1 n(w) )^2}{b_1 n(w)}\right) + n(w) \exp\left(- \frac{(z + a_2 w)^2}{b_2 n(w)} \right) \right]  {\rm d}z
\end{align*}
We bound the first of these terms as follows:
\begin{align*}
& \int_0^\infty 2z\,\exp \Big( - \frac{(z + a_1n(w))^2}{b_1 n(w)}\Big){\rm d}z\\
&= \int_0^\infty {b_1 n(w)} \frac{z}{(z+a_1n(w))} {\rm d} \left( -\exp\left(\frac{-(z+a_1 n(w))^2}{b_1 n(w)}\right)\right) \\
& \leq \int_0^\infty {b_1 n(w)} {\rm d} \left( -\exp\left(\frac{-(z+a_1 n(w))^2}{b_1 n(w)}\right)\right) \\
&= b_1n(w) \exp\left(-\frac{a_1^2 n(w)}{b_1}\right)\,.
\end{align*}
We then  apply similar steps to the second term, which yields
\[
\mathbb{E} \Big[ Z(w)^2 \Big] \leq b_1n(w) \exp\left(-\frac{a_1^2 n(w)}{b_1}\right)+ b_2 n(w)^2 \exp\left(-\frac{a_2^2 n(w)^2}{b_2 w}\right)\,.
\]
The right-hand terms are uniformly bounded in $w$, as required.
\end{proof}

\begin{proof}[Proof of Lemma \ref{hoefd}]
The set ${\mathcal L}$ is assumed to be stable. That is there exists $\delta >0$, $\sigma>0$ and $\kappa>0$ such that 
\begin{equation}\label{StabilityAssumption}
\mathbb E \left[ f(X^{(\lambda)}_k) - f(X^{(\lambda)}_0)\,|\,X^{(\lambda)}_0 = x \right] \leq -\delta \sigma
\end{equation}
for all $k>\sigma$, all $x$ such that $|x|> \kappa$, and $\lambda \in\mathcal L $.

Since \eqref{StabilityAssumption} only occurs after $\sigma$ time units have occurred, we consider our process on steps of size $\sigma$. That is, we consider the process $(X^{(\lambda)}_{\sigma n} : n\in\mathbb Z_+)$. Assuming we start with $x_0 > \kappa$, let $n(x_0)$ be the smallest integer for which $\sigma n(x_0) \geq \tau(x_0)$ holds. 

Since the increments of $f(X^{(\lambda)})$ are bounded by $\phi_f$ we have that
\begin{align}\label{a}
\mathbb P_{x_0} \Big( f(X^{(\lambda)}_{\tau(x_0)}) - f(x_0) \geq z \Big) \leq \mathbb P_{x_0} \Big( f(X^{(\lambda)}_{\sigma n(x_0)} ) - f(x_0) \geq z - \sigma \phi_f \Big)\,.
\end{align}
Let $n_\kappa$ be the hitting time for $(X^{(\lambda)}_{n\sigma} \} : n\in\mathbb Z_+)$ on the states $\{ x : |x|\leq \kappa\}$. 

By splitting the right-hand expression of \eqref{a} into terms depending on whether the event $n_\kappa \leq n(x_0)$ occurs or not, we obtain two terms
\begin{align}
 & {\mathbb P}_{x_0} \left( f(X^{(\lambda)}_{\sigma n(x_0)}) - f(x_0) \geq z - \sigma \phi_f  \right) \notag \\
& =  \,
{\mathbb P}_{x_0} \left( f(X^{(\lambda)}_{\sigma n(x_0)}) - f(x_0) \geq z - \sigma \phi_f  , n_\kappa > n(x_0) \right) \label{f1}\\
&\, +  
{\mathbb P}_{x_0} \left( f(X^{(\lambda)}_{\sigma n(x_0)}) - f(x_0) \geq z - \sigma \phi_f  , n_\kappa \le n(x_0) \right) \label{f2}
\end{align}
We deal with these two terms, \eqref{f1} and \eqref{f2}, separately.

First, we bound the term \eqref{f1}. We consider the process
\[
M_n = f(X^{(\lambda)}_{\sigma n}) - f(x_0) - \delta \sigma n ,
\]
which for times $n$ less than $n_\kappa$ is a supermartingale by the stability assumption \eqref{StabilityAssumption}. Due to our bounded increments assumption we can apply the Azuma--Hoeffding inequality to this process to obtain
\begin{align}
&\mathbb P_{x_0} \Big( f(X^{(\lambda)}_{\sigma n(x_0)}) - f(x_0) \geq z - \sigma \phi_f, n_\kappa > n(x_0) \Big) \notag \\
&\leq \mathbb P_{x_0} \Big( M_{n(x_0)} \geq z - \sigma \phi_f + \delta\sigma n(x_0)   \Big) \leq \exp\left( - \frac{( z - \sigma \phi_f + \delta \sigma n(x_0)   )^2}{2(\phi_f+\delta)^2 \sigma^2 n(x_0)} \right)\,.\label{boundp1}
\end{align}
This provides a bound on our first term \eqref{f1}.

We now bound the second term \eqref{f2}. For this term the process $f(X^{(\lambda)}_{\sigma n})$ has hit below level $\kappa$, so there must be an excursion from level $\kappa$ to level $z+f(x_0)$. There are at most $n(x_0)$ such excursions that can occur from below $z+f(x_0)$. We can apply the Azuma--Hoeffding inequality to each excursion. A simple union bound on these excursions then gives an upper bound that is, for our purposes, sufficiently tight. 

We let $n_0$ be a time for which $\kappa < f(X^{(\lambda)}_{\sigma n_0}) \le \kappa +\sigma\phi_{f}$. We remark that this condition is satisfied immediately after the process leaves the set of states $\{ x : |x| \leq \kappa \}$.
Again let $n_\kappa$ be the first time after $n_0$ for which  $\{ x : |x| \leq \kappa \}$ holds.

Now consider the process
\[
\hat{M}_n = f(X^{(\lambda)}_{\sigma (n \wedge n_\kappa)} - f(X^{(\lambda)}_{\sigma n_0})\,,\quad n\ge0,
\]
which, again by \eqref{StabilityAssumption}, is a supermartingale.

The process $\hat{M}$ follows an excursion of $f(X^{(\lambda)}_{\sigma n})$ from when it hits above $\kappa$ to when it hits below again. 
Further, let $\hat{M}^*_n$ be the maximum achieved by the process $\hat{M}$ by time $n$, that is 
\[
\hat{M}^*_n = \max_{k \le n}\{\hat{M}_k\}\,.
\]
Notice that for the event in \eqref{f2} to hold there must be an excursion of $\hat{M}^*$ from just above $\kappa$ to above $z - \sigma \phi_f + f(x_0)$. We can bound this probability using the Azuma--Hoeffding inequality as follows:
\begin{align*}
\mathbb P_{x_0} \Big( \hat{M}^*_{n(x_0)} \geq z - \sigma \phi_f  + f(x_0)- \kappa\Big) 
\leq & \exp\left( - \frac{( z - \sigma \phi_f + f(x_0)- \kappa )^2}{2\phi_f^2 \sigma^2 n(x_0)} \right)\,.
\end{align*}
Further, there are at most $n(x_0)$ possible excursions of this type. Thus we arrive at the bound
\begin{align}
&\mathbb P_{x_0}  \left( f(X^{(\lambda)}_{\sigma n(x_0)}) - f(x_0) \geq z - \kappa - \sigma \phi_f  , n_\kappa > n(x_0) \right) \notag \\
&\leq  n(x_0)\, \mathbb P_{x_0} \Big( \hat{M}^*_{n(x_0)} \geq z - \sigma \phi_f  + f(x_0)- \kappa\Big)  \notag \\
&\leq  n(x_0) \, \exp\left( - \frac{( z - \sigma \phi_f + f(x_0)- \kappa )^2}{2\phi_f^2 \sigma^2 n(x_0)} \right)\,. \label{boundp2}
\end{align}
Combining the bounds \eqref{boundp1} and \eqref{boundp2}, we find the claimed inequality. 
\end{proof}

\bibliographystyle{acmtrans-ims}
\bibliography{References.bib}
\end{document}